\documentclass[a4paper,reqno,11pt]{amsart}
\usepackage[a4paper,margin=1in]{geometry}
\usepackage{graphicx}
\usepackage{amsmath,amssymb,amsthm}
\usepackage{hyperref}

\usepackage{subcaption}

\theoremstyle{plain}
\newtheorem{theorem}{Theorem}[section]
\newtheorem{lemma}[theorem]{Lemma}
\newtheorem{proposition}[theorem]{Proposition}

\newtheorem{corollary}[theorem]{Corollary}
\theoremstyle{definition}

\numberwithin{equation}{section}
\newcommand{\R}{\mathbb{R}}
\newcommand{\abs}[2][]{#1\lvert #2 #1\rvert}

\newcommand{\FF}{{\mathcal S}}

\title{Global bifurcation and highest waves on water of finite depth}

\author{Vladimir Kozlov$^1$, Evgeniy Lokharu$^1$}
\address{$^1$Department of Mathematics, Link\"oping University, SE-581 83 Link\"oping, Sweden}

\begin{document}
	
\begin{abstract}

We consider the two-dimensional problem for steady water waves  with vorticity on water of finite depth. While neglecting the effects of surface tension we construct connected families of large amplitude periodic waves approaching the limiting wave, which is either a solitary wave, the highest solitary wave, the highest Stokes wave or a Stokes wave with a breaking profile. In particular, when the vorticity is nonnegative we prove the existence of highest Stokes waves with an included angle of 120 degrees. In contrast to previous studies we fix the Bernoulli constant and consider the wavelength as a bifurcation parameter, which guarantees that the limiting wave has a finite depth. In fact, this is the first rigorous proof of the existence of extreme Stokes waves with vorticity on water of finite depth. Beside the existence of highest waves we provide a new result about the regularity of Stokes waves of arbitrary amplitude (including extreme waves). Furthermore, we prove several new facts about steady waves, such as a lower bound for the wavelength of Stokes waves, while also eliminate a possibility of the wave breaking for waves with non-negative vorticity.
	
\end{abstract}

\maketitle

\section{Introduction} \label{s:introduction}

Highest or extreme travelling waves are remarkable objects in the mathematical theory of water waves. The latter are two-dimensional travelling gravity waves enjoying stagnation at every crest, where the relative velocity field vanishes and the surface profile forms a sharp included angle of $120$ degrees. Such extreme waves are referred to as highest waves  because the surface profile at every crest reaches its maximum possible value for a fixed total head (Bernoulli's constant). The first mention of extreme waves is due to Stokes in \cite{Stokes80}, who gave an informal argument showing that the angle of the profile with a simple discontinuity must be $120^\circ$. Thus, Stokes conjectured that there exist nonlinear periodic waves of greatest height enjoying stagnation at every crest whose surface profiles form a sharp angle of $120$ degrees (see Figure 1). This is a remarkable conjecture for the time since then only small-amplitude approximate solutions were known, while extreme waves is a nonlinear phenomenon about waves of large amplitude. 

The existence of periodic wave which are not small is a complicated problem by itself. The first construction of large-amplitude waves was given by Krasovskii in \cite{Krasovskii60} who proved that for any given flux $Q$, period $L$ and $\beta \in (0,\tfrac{\pi}{6})$ there exists a Stokes wave with $\max \theta = \beta$, where $\theta$ is the inclination angle to the horizontal. Even so the waves found by Krasovskii are of large amplitude it was not clear if they are close to stagnation or not. Later, Keady and Norbury \cite{KeadyNorbury78} showed by applying the global bifurcation theory to the Nekrasov equation that there exist Stokes waves which are arbitrary close to the stagnation. That is for any speed at the crest $q_c \in (0,c)$, period $L$ and flux $Q$ there exists the corresponding Stokes wave; here $c$ is the wave speed. Now one could think of constructing a highest wave by passing to the limit. This was done by Toland \cite{Toland1978a} who proved the existence of highest Stokes waves with surface stagnation, partly verifying the Stokes conjecture in the infinite depth case. For waves on finite depth this was established by McLeod \cite{McLeod1997} (results were obtained around 1979 and published in 1997), who also showed that the inclination angle might overcome $\tfrac16 \pi$ for waves close to the stagnation. Mentioned results for periodic waves were extended to solitary waves by Amick and Toland \cite{AmickToland81a, AmickToland81b}, who in particular obtained the existence of the highest solitary wave. 

Though highest Stokes and solitary waves in the irrotational setting were already constructed in 1981, the question about their regularity and geometry remained open. It was settled independently by Amick, Fraenkel and Toland in \cite{AmickFraenkelToland82} and by Plotnikov \cite{Plotnikov82}. They proved that any highest wave forms a sharp included angle of $120^\circ$ at every crest just as predicted by Stokes. Later Fraenkel gave another, more direct and constructive proof in  \cite{Fraenkel2006}. Beside, Plotnikov and Toland \cite{Plotnikov2004} verified the remaining part of the Stokes conjecture about the convexity of highest waves. Finally, Amick and Fraenkel \cite{Amick1987a} obtained an asymptotic expansion of a periodic highest wave near a stagnation point.

The classical papers mentioned above concern with the water wave problem for Stokes and solitary waves, though it is known that other types of waves exist. See for instance \cite{Zufiria1987, Baesens1992} and a recent study \cite{BuffoniDancerToland00a} on subharmonic bifurcations. For an analysis of singularities of such waves one needs more general methods developed recently by Varvaruca and Weiss \cite{Varvaruca2011}. They proved that irrespectively of the wave geometry all stagnation points are isolated, while the surface profile must form there an included angle of $120$ degrees. 

The rotational theory is much less developed. The first construction of large-amplitude waves is due to Constantin and Strauss in \cite{ConstantinStrauss04}. Using the degree theory and a continuation argument they proved that for an arbitrary vorticity distribution there exist continuous families of Stokes waves with a fixed period, bifurcating from a uniform stream and approaching "stagnation". Essentially, they show that there exists a sequence of Stokes waves such that $u \to c$ along the sequence, where $u$ is the horizontal component of the velocity field and $c$ is the wave speed. Thus, the limiting wave, if exists, has a critical point ($u = c$) somewhere in the fluid. We find a stagnation point if this happens at the crest, but for rotational waves this is not always the case. Furthermore, depending on the sign of the vorticity, there exist laminar flows with stagnation points either everywhere on the surface or at the bottom. Thus, the limiting wave might be trivial. Another possibility is that the Bernoulli constant becomes unbounded and waves converge to the highest wave but on the infinite depth. The same problem appears in the subsequent papers by Constantin and Strauss \cite{Constantin_2011}, Varvaruca \cite{Varvaruca09}, Constantin, Strauss and Varvaruca \cite{Constantin2016} and Varholm \cite{Varholm2019}. Thus, the existence of highest waves with vorticity is still an open problem.

The presence of laminar flows with surface stagnation shows that the geometry of highest waves with vorticity might be different, compared to the irrotational setting, when the surface profile always forms an angle of $120^\circ$. However, as shown by Varvaruca \cite{Varvaruca09} and later by Varvaruca \& Weiss \cite{Varvaruca2012}, there are only two possibilities: the Stokes corner of $120^\circ$ or a horizontally flat stagnation. We shall emphasize that it is not known if the second options is really possible for waves other than laminar flows.

In the present paper we construct highest Stokes waves with a favourable vorticity distribution, while also prove a general statement for an arbitrary vorticity distribution, which is of separate interest. Beside the existence, we also prove some new facts about steady waves. In Theorem 3.4 we give a lower bound for the wavelength, depending on the Bernoulli constant. A possibility of the wave breaking is eliminated in Proposition 3.2 for non-negative vorticity distributions, which is related to results of \cite{StraussWheeler16}. Furthermore, we obtain a new regularity statement for Stokes waves in the irrotational setting, see our Theorem 2.4. It is remarkable that the claim is valid even for extreme waves.

\begin{figure}[t!]
	\centering
	\begin{subfigure}[t]{0.5\textwidth}
		\centering
		\includegraphics[scale=0.8]{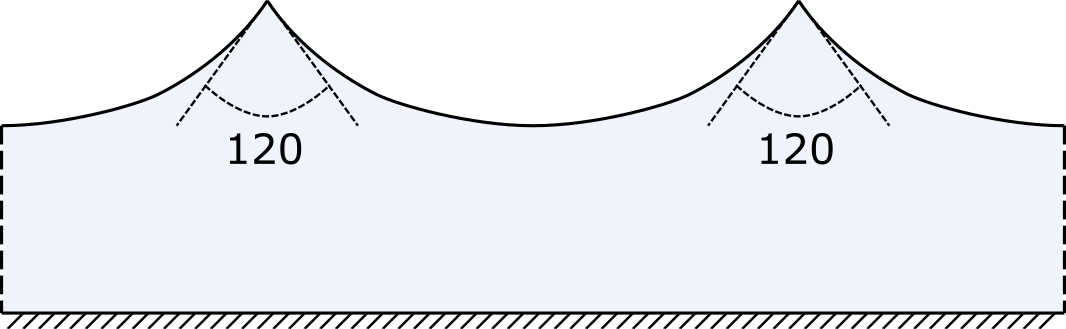}
		\caption{an extreme Stokes wave}
		%\label{fig:H1}
	\end{subfigure}%
	~ 
	\begin{subfigure}[t]{0.5\textwidth}
		\centering
		\includegraphics[scale=0.8]{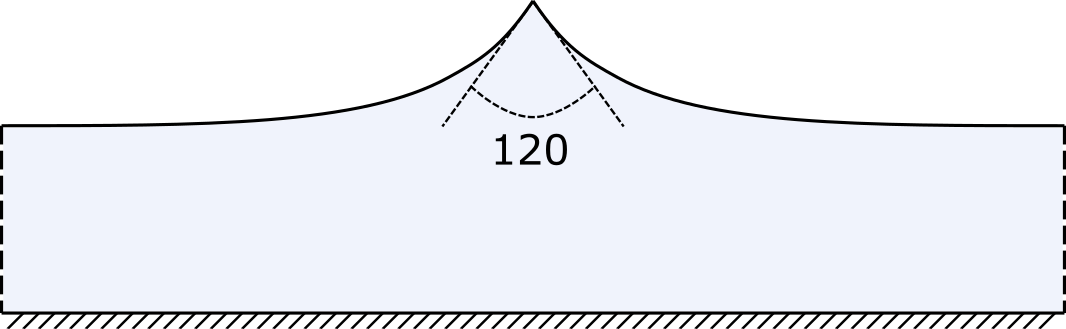} \caption{an extreme solitary wave}
		%\label{fig:phase}and
	\end{subfigure}
	\caption{}
\end{figure}

\section{Statement of the problem} \label{s:introduction}

We consider the classical water wave problem for two-dimensional steady waves with vorticity on water of finite depth. We neglect effects of surface tension and consider an ideal fluid of constant (unit) density. Thus, in an appropriate coordinate system moving along with the wave, stationary Euler equations are given by
\begin{subequations}\label{eqn:trav}
	\begin{align}
	\label{eqn:u}
	(u-c)u_x + vu_y & = -P_x,   \\
	\label{eqn:v}
	(u-c)v_x + vv_y & = -P_y-g, \\
	\label{eqn:incomp}
	u_x + v_y &= 0, 
	\end{align}
	which holds true in a two-dimensional fluid domain $D_\eta$, defined by the inequality
	\[
	0 < y < \eta(x).
	\]
	Here $(u,v)$ are components of the velocity field, $y = \eta(x)$ is the surface profile, $c$ is the wave speed, $P$ is the pressure and $g$ is the gravitational constant. The corresponding boundary conditions are
	\begin{alignat}{2}
	\label{eqn:kinbot}
	v &= 0&\qquad& \text{on } y=0,\\
	\label{eqn:kintop}
	v &= (u-c)\eta_x && \text{on } y=\eta,\\
	\label{eqn:dyn}
	P &= P_{\mathrm{atm}} && \text{on } y=\eta.
	\end{alignat}
\end{subequations}
It is often assumed in the literature that the flow is irrotational, that is $v_x - u_y$ is zero everywhere in the fluid domain. Under this assumption components of the velocity field are harmonic functions, which allows to apply methods of complex analysis. Being a convenient simplification it forbids modeling of non-uniform currents, commonly occurring in nature. In the present paper we will consider rotational flows, where the vorticity function is defined by
\begin{equation} \label{vort}
\omega = v_x - u_y.
\end{equation}
Throughout the paper we assume that the flow is free from stagnation points and the horizontal component of the relative velocity field does not change sign, that is
\begin{equation} \label{uni}
u-c < 0 
\end{equation}
everywhere in the fluid. We call such flows unidirectional.
% TODO : problems in the bibliography with (at least)

In the two-dimensional setup relation \eqref{eqn:incomp} allows to reformulate the problem in terms of a stream function $\psi$, defined implicitly by relations
\[
\psi_y = c-u, \ \ \psi_x =  v.
\]
This determines $\psi$ up to an additive constant, while relations \eqref{eqn:kinbot},\eqref{eqn:kinbot} force $\psi$ to be constant along the boundaries. Thus, by subtracting a suitable constant, we can always assume that
\[
\psi = m, \ \ y = \eta; \ \ \psi = 0, \ \ y = 0.
\]
Here $m$ is the mass flux, defined by
\[
m = \int_0^\eta (u-c) dy.
\]
In what follows we will use non-dimensional variables proposed by Keady \& Norbury \cite{KeadyNorbury78}, where lengths and velocities are scaled by $(m^2/g)^{1/3}$ and $(mg)^{1/3}$ respectively; in new units $m=1$ and $g=1$. For simplicity we keep the same notations for $\eta$ and $\psi$.

Taking the curl of Euler equations \eqref{eqn:u}-\eqref{eqn:incomp} one checks that the vorticity function $\omega$ defined by \eqref{vort} is constant along paths tangent everywhere to the relative velocity field $(u-c,v)$; see \cite{Constantin11b} for more details. Having the same property by the definition, stream function $\psi$ is strictly monotone by \eqref{uni} on every vertical interval inside the fluid region. These observations together show that $\omega$ depends only on values of the stream function, that is
\[
\omega = \omega(\psi).
\]
This property and Bernoulli's law allow to express the pressure $P$ as
\begin{align}
\label{eqn:bernoulli}
P-P_\mathrm{atm} + \frac 12\abs{\nabla\psi}^2 + y  + \Omega(\psi) - \Omega(1) = const,
\end{align}
where 
\begin{align*}
\Omega(\psi) = \int_0^\psi \omega(p)\,dp
\end{align*}
is a primitive of the vorticity function $\omega(\psi)$. Thus, we can eliminate the pressure from equations and obtain the following problem:

	\begin{subequations}\label{sys:stream}
	\begin{alignat}{2}
	\label{sys:stream:lap}
	\Delta\psi + \omega(\psi)&=0 &\qquad& \text{for } 0 < y < \eta,\\
	\label{sys:stream:bern}
	\tfrac 12\abs{\nabla\psi}^2 +  y  &= r &\quad& \text{on }y=\eta,\\
	\label{sys:stream:kintop} 
	\psi  &= 1 &\quad& \text{on }y=\eta,\\
	\label{sys:stream:kinbot} 
	\psi  &= 0 &\quad& \text{on }y=0.
	\end{alignat}
The assumption \eqref{uni} transforms into 	
\begin{equation}\label{uni:psi}
	\psi_y > 0 \ \ \text{for } 0 \leq y \leq \eta.
	\end{equation}	
The latter forbids the presence of stagnation points, in particular surface stagnation points that occur when $\eta(x) = r$ for some $x \in \R$. As for regularity, we assume that $\psi \in C^{2,\gamma}(\overline{D_\eta})$ and $\eta \in C^{2,\gamma}(\R)$, provided $\omega \in C^{1,\gamma}([0,1])$ for some $\gamma \in (0,1)$. The H\"older exponent $\gamma$ will remain unchanged in what follows. We will also consider highest waves, for which $\max \eta = r$. Such waves enjoy stagnation at every crest, where $\psi_y = \psi_x = 0$. As we will show below, highest waves are $C^{2,\gamma}$ everywhere, except the crest points, where the stream function $\psi$ is $C^{1}$-continuous, while the surface profile $\eta$ is of the class $C^{1}$ on each closed interval without stagnation points inside (though it can be a boundary point of the interval). In the irrotational case the regularity of the stream function can be significantly improved (Theorem \ref{thm:reg}).\\

In what follows we will consider families of Stokes waves, periodic solutions to \eqref{sys:stream} that are monotone between each neighbour crest and trough and symmetric around every vertical line passing through crests and troughs; see \cite{Constantin2007, Hur07} for related results on the symmetry of Stokes waves.

In our setup the Bernoulli constant $r$ is fixed, though different values of $r$ lead to different bifurcation curves and different highest waves. There are several advantages to have $r$ fixed: 
	
	\begin{itemize}
		\item first of all, it guarantees that the mean depth is separated from zero and bounded from above by constants depending on $r$; this resolves the main drawback of previous constructions, where the Bernoulli constant could become unbounded, while waves are converging to the highest wave on infinite depth;
		\item secondly, we have an improved uniform regularity (including the highest waves) along the whole bifurcation curve; see our Theorem \ref{thm:reg}, which states that there is a H\"older exponent $\alpha = \alpha(r)$, depending only on $r$, such that all stream functions (of Stokes waves) and their derivatives are uniformly H\"older continuous, that is norms $\|\psi\|_{C^{1,\alpha}(\overline{D_\eta})}$ are uniformly bounded by a constant determined by $r$; in particular this specifies the type of convergence of regular Stokes to the highest wave;
		\item there is no stagnation occur inside or at the bottom; 
		\item finally, this approach allows to construct solitary waves and even extreme solitary waves.
	\end{itemize}
	
The Bernoulli constant $r$ in \eqref{sys:stream:bern} can not be arbitrary. It was shown in \cite{Kozlov2012, Kozlov2015} that $r >R_c$ for any non-laminar solution, where $R_c$ is the minimal Bernoulli constant among all laminar flows (represented by solutions to \eqref{sys:stream}, with $\psi = U(y)$ and $\eta = const$). To simplify our considerations we will also assume that $r < d_0$, where $d_0$ is the maximal depth among all laminar solutions. In the irrotational case $d_0 = +\infty$, so this is not a restriction. In fact, it only matters when the vorticity function $\omega$ is negative; see \cite{Kozlov2015}.

\end{subequations}	

%\subsection{Admissible vorticity distributions}

%In what follows we will consider vorticity functions $\omega \in C^{1,\gamma}([0,1])$ for which the water wave problem \eqref{sys:stream} has the following property: there exists a constant $M = M(\omega)$ such that every Stokes wave solution $(\psi,\eta)$ of \eqref{sys:stream} satisfies $|\eta'(x)| \leq M$ for all $x \in \R$. For instance, $\omega = 0$ is admissible, which follows from \cite{Amick1987}. On the other hand, any constant positive vorticity is also admissible as shown in \cite{StraussWheeler16}.

\subsection{Partial hodograph transform}

Under assumption \eqref{uni:psi} we can apply the partial hodograph transform introduced by Dubreil-Jacotin in \cite{DubreilJacotin34}. Thus, we present new independent variables
\[
q = x, \ \ p = \psi(x,y),
\]
while new unknown function $\hat{h}(q,p)$ (height function) is defined from the identity
\[
\hat{h}(q,p) = y.
\]
Note that it is related to the stream function $\psi$ through the formulas
\begin{equation} \label{height:stream}
\psi_x = - \frac{\hat{h}_q}{\hat{h}_p}, \ \ \psi_y = \frac{1}{\hat{h}_p},
\end{equation}
where 
\begin{subequations}\label{sys:h}
	\begin{equation} \label{sys:h:uni}
	\hat{h}_p > 0
	\end{equation}
	throughout the fluid domain by \eqref{uni:psi}. An advantage of using new variables is in that instead of two unknown functions $\eta(x)$ and $\psi(x,y)$ with an unknown domain of definition, we have one function $\hat{h}(q,p)$ defined in a fixed strip $S = \R \times (0,1)$. An equivalent problem for $\hat{h}(q,p)$ is given by
	\begin{alignat}{2}
	\label{sys:h:main}
	\left( \frac{1+\hat{h}_q^2}{2\hat{h}_p^2} + \Omega \right)_p - \left(\frac{\hat{h}_q}{\hat{h}_p}\right)_q &=0 &\qquad& \text{in } S,\\
	\label{sys:h:top}
	\frac{1+\hat{h}_q^2}{2\hat{h}_p^2} +  \hat{h}  &= r &\quad& \text{on }p=1,\\
	\label{sys:h:bot} 
	\hat{h} &= 0 &\quad& \text{on }p=0.
	\end{alignat}
\end{subequations}
The wave profile $\eta$ is recovered from the boundary value of $\hat{h}$ on $p = 1$:
\[
\eta(q) = \hat{h}(q,1), \ \ q \in \R.
\]
This formulation is in fact equivalent to \eqref{sys:stream}; see \cite{ConstantinStrauss04} for details.

\subsection{Stream solutions}

Uniform streams, flows with flat surface and velocity field depending only on the depth, are represented by stream solutions $H(p;s)$ of \eqref{sys:h}. The latter depend only on $p$-variable and are parametrized by a parameter $s>s_0$, which equals to the velocity of the flow at the bottom (in the original variables). Solving \eqref{sys:h:main} for $H$ explicitly, one finds
\[
H(p;s) = \int_0^p \frac{1}{\sqrt{s^2- 2 \Omega(p')}} dp',
\]
where $\Omega(p) = \int_0^p \omega(p') dp'$ is a primitive of the vorticity function. Note that $s^2 \geq s_0^2 = \max_{p \in [0,1]} 2 \Omega(p)$. The corresponding to $H(p;s)$ depth and the Bernoulli constant $R(s)$ are given explicitly by
\[
d(s) = \int_0^1 \frac{1}{\sqrt{s^2- 2 \Omega(p')}} dp', \ \ R(s) = \tfrac12s^2- \Omega(1)  + d(s).
\]
Note that $d(s) < d_0:=d(s_0)$ for all $s \in (s_0,+\infty)$, where $d_0$ might be finite or not, depending on the vorticity function; for instance, $d_0< + \infty$ when $\Omega(p) < 0$ for all $p \in (0,1]$ as well as $\omega(0) < 0$ (negative vorticity), while $d_0 = +\infty$ for $\omega = 0$. It is clear that constants $R_0 = R(s_0)$ and $d_0$ are finite or not at the same time. 

It is shown in \cite{Kozlov2015} that $R(s)$ attains its minimum $R_c$ for some $s_c > s_0$, determined from
\[
\int_0^1 \frac{1}{(s_c^2- 2 \Omega(p'))^{\tfrac32}} dp' = 1.
\]
As shown in \cite{Kozlov2015} any steady motion, other than a uniform stream, has $r > R_c$. Furthermore, it was shown that for "positive" vorticity functions we also have $r < R_0$ whenever $R_0$ is finite. On the other hand, for negative vorticity functions there is a weaker inequality $r < R_0 - \Omega(1)$, obtained recently in \cite{Lokharu2020}. The bound $r < R_0$ is important because it guarantees that there exist so-called subcrtitical and supercritical uniform streams corresponding to $r$. The latter are defined by parameters $s$ that solve equation $R(s) = r$, which possesses exactly two roots $s = s_-(r)$ and $s = s_+(r)$ with $s_-(r) < s_c <  s_+(r)$, provided $r > R_c$ and $r < R_0$. The corresponding depths are $d_+(r) = d(s_-(r))$ and $d_-(r) = d(s_+(r))$. The existence of a subcritical stream is essential for our purposes, since only subcritical streams with $d = d_+(r)$ and $s = s_-(r)$ admit a local bifurcation curves of small-amplitude Stokes waves; see \cite{Kozlov2014, Kozlov2017a, Groves2008}. Note that our assumption $r < d_0$ from Theorem \ref{thm:vort} guarantees that $r < R_0$ and the subcritical stream is well-defined. Another purpose of $r < d_0$ is that it forbids solutions along the bifurcation curve from approaching stagnation at the bottom. A precise statement is given in Proposition \ref{p:psiy:bot} and can be expressed in a form of inequalities
\[
\hat{s} \leq \psi_y(x,0) \leq \check{s}, \ \ x \in \R,
\]
where $\hat{s}$ and $\check{s}$ are defined from the identities $d(\hat{s}) = \max \eta$ and  $d(\check{s}) = \min \eta$. In particular we require that there exists a uniform stream with $d = \max \eta$, which is ensured by the assumption $r < d_0$.

It is worth mentioning that inequality $r < d_0$ is believed to be true for all steady waves other than streams. It is known to be true for a large class of vorticity functions (classes (i) and (iii) according to the classification in \cite{Kozlov2015}) except vorticity functions $\omega$, for which $\Omega(p) < 0$ on $(0,1]$ and $\omega(0) < 0$. We believe that in this case the statement is also true, even so a rigorous justification is not available at the moment.

Let us recall another fundamental bound for steady waves from \cite{Kozlov2015} (see also \cite{Kozlov2007, Kozlov2009, Kozlov2012, Keady1975, KeadyNorbury78}). It is known that an arbitrary solution, other than a uniform stream, is subject to a strict inequality
\begin{equation} \label{boundeta}
\max\eta > d_+(r),
\end{equation}
where $d_+(r)$ is the depth of the subcritical laminar flow. When the Bernoulli constant $r$ is fixed, relation \eqref{boundeta} determines a positive cone of functions; this fact will be used in our proof of Theorem \ref{thm:vort}.

\subsection{Problem with a fixed period}

Assume that $\hat{h}$ has period $\Lambda$ with respect to $q$-variable. Then we put
\[
h(q,p) = \hat{h}(\Lambda(2\pi)^{-1} q ,p),
\]
which is $2\pi$-periodic function that solves 
\begin{subequations}\label{sys:hs}
	\begin{alignat}{2}
	\label{sys:hs:main}
	\left( \frac{1+\lambda^2 h_q^2}{2h_p^2} + \Omega \right)_p - \lambda^2 \left(\frac{h_q}{h_p}\right)_q &=0 &\qquad& \text{in } S,\\
	\label{sys:hs:top}
	\frac{1+\lambda^2 h_q^2}{2 h_p^2} +  h  &= r &\quad& \text{on }p=1,\\
	\label{sys:hs:bot} 
	h &= 0 &\quad& \text{on }p=0,
	\end{alignat}
where $\lambda = 2\pi \Lambda^{-1}$. As before, we require
\begin{equation} \label{sys:hs:uni}
h_p > 0.
\end{equation}	
\end{subequations}
We will consider this problems for $h \in C^{2,\gamma}_{per}(\overline{S})$, where the latter is a subspace of $C^{2,\gamma}(\overline{S})$, that consists of all even functions with period $2\pi$. In what follows $\lambda = 2\pi\Lambda^{-1}$ will play the role of the bifurcation parameter.

\subsection{Statements of main results}

\begin{theorem}[Arbitrary vorticity]\label{thm:vort}
	Assume that $\omega \in C^{1,\gamma}([0,1])$. Then for any $r \in (R_c,d_0)$ there exist functions ${\mathcal C}:[0,+\infty) \to C^{2,\gamma}_{per}(\overline{S})$ and $\lambda:[0,+\infty) \to (0,+\infty)$ with the following properties:
	
	\begin{itemize}
		\item[(i)] ${\mathcal C}(0)$ is the subcritical stream solution $H(p;s_+(r))$;
		\item[(ii)] ${\mathcal C}(t) = h(q,p; t), t>0$ is a solution to \eqref{sys:hs} with $\lambda = \lambda(t)$; the corresponding solution $(\psi,\eta)$ represents a Stokes wave with the period $\Lambda = \lambda(t)^{-1}$;
		\item[(iii)] $({\mathcal C}(t),\lambda(t))$ has a real analytic reparametrization locally around each $t \geq 0$.
		
	\end{itemize}
	Furthermore, there exists a sequence $t_j \to +\infty$ such that we either have
	\begin{itemize}
		\item[(I)] $\max_{x\in\R}\eta(x;t_j) \to r$ as $j \to +\infty$ and $\Lambda(t_j) \to \Lambda < + \infty, \Lambda \neq 0$ (stagnation at every crest);
		\item[(II)]  $\lim_{j \to +\infty}\max_{x\in\R}\eta(x;t_j) < r$ and $\Lambda(t_j) \to +\infty$ as $j \to +\infty$ (solitary wave);
		\item[(III)]  $\max_{x\in\R}\eta(x;t_j) \to r$  and $\Lambda(t_j) \to + \infty$ as $j \to +\infty$ (extreme solitary wave);
		\item[(IV)] $\max_{x \in \R} |\eta'(x;t_j)| \to +\infty$  as $j \to +\infty$ (wave breaking).
	\end{itemize}
	Here $\eta(x;t)$ is the surface profile corresponding to $h(q,p;t)$, $t \geq 0$. 
\end{theorem}

\begin{corollary}[Existence of highest waves] \label{cor:vort} Assume that $\omega \geq 0$, then there exists either 
	\begin{itemize}
		\item[(I')] an extreme Stokes wave $(\psi,\eta)$ with $\psi \in C^1(\overline{D_\eta})$, enjoying stagnation at every crest (and nowhere else), where $\eta$ forms either a sharp angle of 120 degrees or included angle of $180^{\circ}$; furthermore, $\eta$ is locally in $C^{2,\gamma}$ everywhere outside the stagnation points, while it is in $C^1$ on every closed interval that does not contain any crest in it's interior.
		\item[(II')] or a solitary wave solution $(\psi,\eta)$ with $\psi \in C^{2,\gamma}(\overline{D_\eta})$ and $\eta \in C^{2,\gamma}(\R)$;
		\item[(III')] or the highest solitary wave solution $(\psi,\eta)$ with $\psi \in C^{1}(\overline{D_\eta})$, enjoying stagnation at the crest (and nowhere else), where $\eta$ forms either a sharp angle of 120 degrees or included angle of $180^{\circ}$;  furthermore, $\eta$ is locally in $C^{2,\gamma}$ everywhere outside the crest, while it is in $C^1$ on every closed interval that does not contain the stagnation point in it's interior.		
	\end{itemize}
	Furthermore, there is a constant $r_\star \in (R_c,d_0]$ (depending only on $\omega$) such that only the first option (I') can occur for $r \in (r_\star,d_0)$. When $r$ is close to $r_c$ only the second option is possible, which leads to a solitary wave.
\end{corollary}

Note that the constant $r_\star$ from Corollary \ref{cor:vort} is explicit in the following sense. It corresponds to the bound for the Froude number for solitary waves with vorticity. Indeed, we need to guarantee that no solitary waves exist for $r > r_\star$ (or for $F>F_\star$). For positive vorticity as in Corollary \ref{cor:vort} the known bound is $F<2$ (see \cite{Wheeler15b}), which corresponds to $r_\star$ with $F_\star = 2$. Thus, we would find that $r_\star < d_0$ (strictly less) for small and moderate positive vorticity functions. For large positive vorticity it is only known that no smooth solitary waves exist for $r \geq r_0$ (in this case $r_\star = d_0$). It is not known if there exists a family of small-amplitude solitary waves converging to a uniform stream with surface stagnation (or that the "highest" solitary wave coincides with the uniform stream with surface stagnation). Having that in mind we can not always guarantee that $r_\star$ is strictly less than $d_0$ so that one can eliminate possibilities of (II') and (III').

In the irrotational case, when $R_c = \tfrac32$ and $d_0 = +\infty$ the statement above can be improved as follows.

\begin{corollary}[Existence of highest waves in the irrotational setting] \label{cor:irr} For any $r > \tfrac32$ there exists either
	\begin{itemize}
		\item[(I$^\circ$)] an extreme Stokes wave $(\psi,\eta)$ with $\psi \in C^1(\overline{D_\eta})$, enjoying stagnation at every crest (and nowhere else), where $\eta$ forms a sharp angle of 120 degrees;  furthermore, $\eta$ is locally in $C^{2,\gamma}$ everywhere outside the stagnation points, while it is in $C^1$ on every closed interval that does not contain any stagnation point in it's interior.
		\item[(II$^\circ$)] or a solitary wave solution $(\psi,\eta)$ with $\psi \in C^{2,\gamma}(\overline{D_\eta})$ and $\eta \in C^{2,\gamma}(\R)$;
		\item[(III$^\circ$)] or the highest solitary wave solution $(\psi,\eta)$ with $\psi \in C^{1}(\overline{D_\eta})$, enjoying stagnation at the crest (and nowhere else), where $\eta$ forms a sharp angle of 120 degrees; furthermore, $\eta$ is locally in $C^{2,\gamma}$ everywhere outside the crest, while it is in $C^1$ on every closed interval that does not contain the stagnation point in it's interior.			
	\end{itemize}
	Furthermore, only the first option (I$^\circ$) can occur for $r > 2^{2/3} \approx 1.587$.
\end{corollary}

Note that the alternative of a flat stagnation as in Corollary \ref{vort} can not occur for irrotational waves and more general for waves with negative vorticity (around the surface); see \cite{Varvaruca09}. \\

\begin{figure}[t!] \label{fig:diag}
	\centering
	\includegraphics[scale=1]{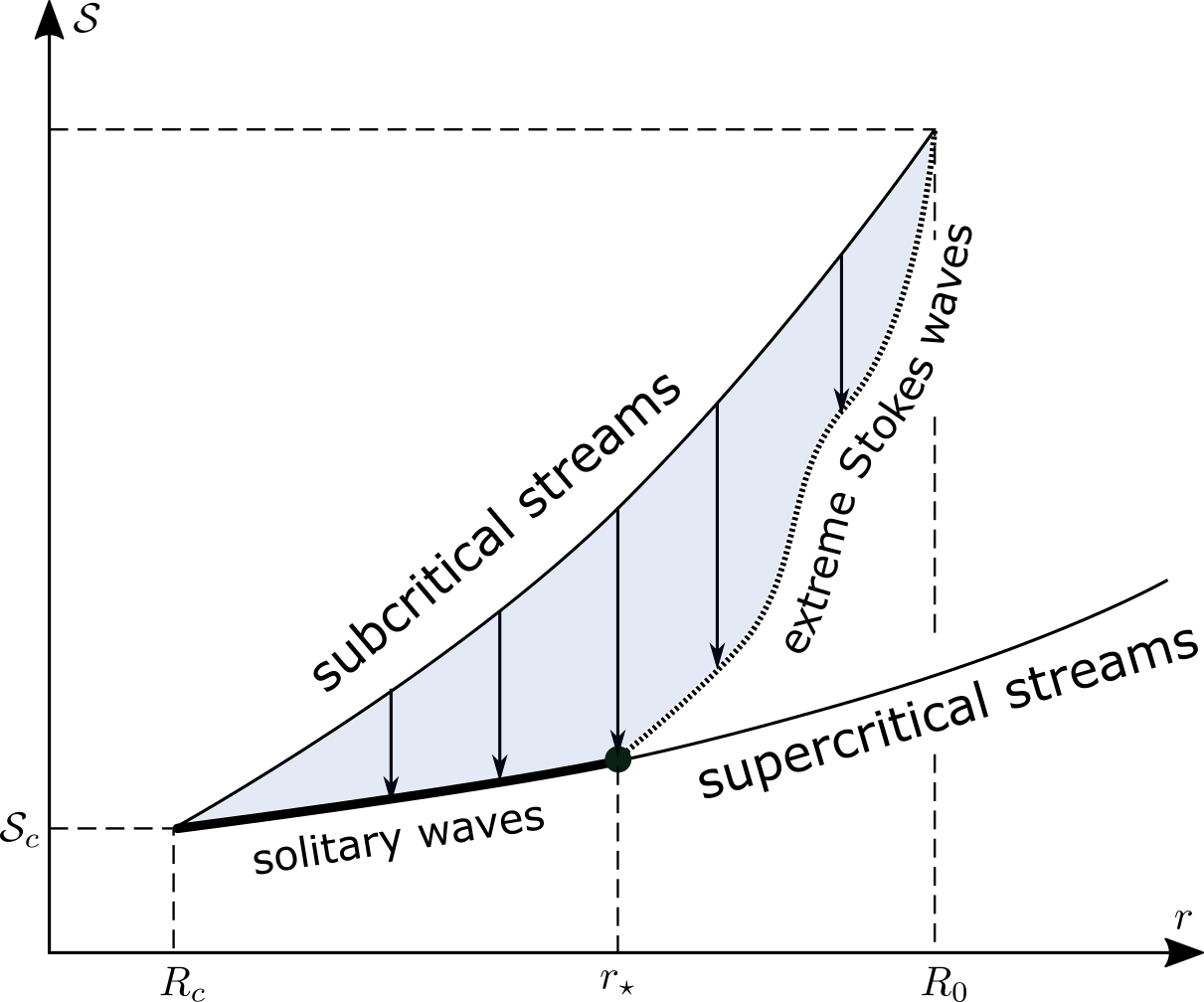}
	\caption{A bifurcation diagram in $(r,\FF)$ coordinates; bifurcation curves correspond to vertical lines with arrows.}
\end{figure}

There is a convenient way to visualize and discuss our results in the $(r,\FF)$-plane, where $\FF$ is the flow force constant. For this purpose we refer to Figure 2, where we sketched a cuspidal region from the Benjamin and Lighthill conjecture; see \cite{Benjamin95}. The cusp in the figure consists of two curves representing parallel flows. Subcritical flows on the upper curve give rise to bifurcations of small-amplitude Stokes waves, while the lower curve represents supercritical streams, supporting solitary waves; see \cite{KozLokhWheeler2020}. It was recently proved in \cite{Lokharu2020b} that all steady waves with vorticity correspond to points inside the cuspidal region (the Benjamin and Lighthill conjecture). It has long been known that not all points from the cusp represent steady waves; Benjamin and Lighhill conjectured that there must a third barrier (similar to the dashed line in the figure) corresponding to the highest waves. This was partly verified by Cokelet \cite{Cokelet1977a}, who found that such barrier exists but consists of waves that are close to stagnation, not necessarily the highest waves. Nevertheless it is believed that there must be a curve of highest waves that starts from the highest solitary wave, constructed in the irrotational setting by Amick and Toland in \cite{AmickToland81b}. 

It follows from our Theorem \ref{thm:vort} (and from Corollary \ref{cor:vort}) that for a given $r > R_c$ there is a bifurcation curve of Stokes waves that starts from a subcritical uniform stream, corresponding to a point on the upper boundary of the cusp in Figure 2. The whole bifurcation curve consists of solutions having the same Bernoulli constant $r$ and therefore corresponding to a horizontal interval (arrows) in the diagram. When $r$ is close to $r_c$ the bifurcation curve necessarily reaches a solitary wave (represented by a point on the solid curve on the bottom boundary of the cusp) as shown in \cite{Kozlov2017a}. For all large values of $r$ (when $r > r_\star$), as follows from Corollary \ref{cor:vort}, we always reach an extreme Stokes wave, corresponding to the dashed curve in the diagram. Thus, we can conjecture that the minimum value of $r_\star$ with this property shall correspond to the highest solitary wave (the solid circle in Figure 2), which we expect to exist for all nonnegative vorticity functions.

Beside the existence of highest waves we present a new result about uniform regularity of arbitrary Stokes waves (including extreme waves).

\begin{theorem}[Regularity]\label{thm:reg} For any $r > \tfrac32$ there exist constants $\alpha \in (0,\tfrac12)$ and $C_1>0$ such that any Stokes wave solution to \eqref{sys:stream} with $\omega = 0$ is subject to
	\begin{equation}\label{psi:reg}
	\|\psi\|_{C^{1,\alpha}(\overline{D_\eta})} \leq C_1,
	\end{equation}
	where both $C_1$ and $\alpha$ depend only on $r$. Furthermore, there is a constant $C_2> 0$ such that
	\begin{equation}\label{psi:rminy}
	C_2(r-y) \leq \psi^2_y(x,y) \leq 2 (r-y)
	\end{equation}
	for all $x,y \in D_\eta$, while $C_2$ depends only on $r$.
\end{theorem}

This regularity result is closely related to the Harnack principle for solutions of the Poisson problem (with the right-hand side), recently studied in \cite{Allen2019}. Unfortunately, our result does not follow from \cite{Allen2019} since their theorem is only valid 
in domains with small Lipschitz constants, which might not be the case here. Thus, we present an alternative approach, based on inequalities \eqref{psi:rminy}. As we will show below, \eqref{psi:rminy} holds for every solution (not necessarily a Stokes wave) for which $M:=\sup_{\R}|\eta'|$ is bounded by a constant strictly less than $1$, while the resulting constant $C_2$ in \eqref{psi:rminy} depends on the difference $1-M$.

\section{Preliminaries} \label{s:introduction}

\subsection{Bounds for the velocity field}

The next proposition guarantees that flows with a fixed Bernoulli constant $r$ are away from stagnation at the bottom.

\begin{proposition}\label{p:psiy:bot} Let $r \in (R_c,d_0)$ and $(\psi,\eta) \in C^{2,\gamma}(\overline{D_\eta}) \times C^{2,\gamma}(\R)$ is a solution to \eqref{sys:stream}. Then 
\begin{equation} \label{eq:psiy:bot}
	U_y(0;\max\eta) \leq \psi_y(x,0) \leq U_y(0;\min\eta),
\end{equation}
where $(U(y;d),d)$ is the stream solution of depth $d$.
\end{proposition}
\begin{proof}
	This statement can be proved in terms of height functions introduced in Section 2.3. Let $\hat{h}(q,p)$ be the height function corresponding to $(\psi,\eta)$ and $\hat{H}(p)$ is the one for $U(y;d)$ with $d = \hat{\eta}$. Note that we can find such stream solution $U(y;d)$ because of the assumption $r < d_0$. Now we consider the difference $w = \hat{h}-\hat{H}$, which solves a homogeneous elliptic problem in $S$, zero on the bottom and non-positive along the upper boundary $p=1$. Thus, $w \leq 0$ in $S$ and then $w_p(q,0) \leq 0$ by the Hopf lemma. This gives the lower bound in \eqref{eq:psiy:bot}. The remaining inequality can be proved in the same way.
\end{proof}

Note  that $\max \eta \leq r$ by the Bernoulli equation \eqref{sys:stream:bern} and then $U_y(0;\hat{\eta}) \geq U_y(0;r) > 0$, which shows that $\psi_y$ at the bottom is separated from zero by a constant depending only on $r$.

When a solution is locally away from stagnation, one can prove the following regularity property, valid even for highest waves but outside the stagnation points.

\begin{proposition} \label{p:delta} Assume that $\omega \in C^{1,\gamma}([0,1])$. Let $r \in (R_c,d_0)$ and $\delta>0$ be given as well as a ball $B$ of radius $\rho > 0$ and let $M = \sup_{(x,\eta(x)) \in B} |\eta'(x)|$. Then there exist constants $\alpha \in (0,1)$ and $C>0$, depending only on $r,\delta,\rho$ and $M$ such that any solution $(\psi,\eta) \in C^{2,\gamma}(\overline{D_\eta}) \times C^{2,\gamma}(\R)$ with $\inf_{B} \psi_y \geq \delta $ satisfies $\|\psi\|_{C^{3,\alpha}(D_\eta \cap \tfrac12 B)} \leq C$, where $\tfrac12 B$ is a ball with the same centre and radius $\tfrac12 \rho$.
\end{proposition}

This statement can be proved just as \cite[Proposition 4.2]{Kozlov2017a}. Note that when $\inf \psi_y > 0$ the regularity of $\psi$ is only limited by the regularity of the vorticity $\omega$. 

\subsection{Estimates at the boundary}

The next technical result is essential in establishing properties of highest waves as in Corollary \ref{cor:vort} and for proving claims of Theorem \ref{thm:vort}.

\begin{proposition} \label{thm:slope} For any $\omega \in C^{1,\gamma}([0,1])$ and any $r > R_c$ there exists a constant $C_1 > 0$ such that any Stokes wave solution $(\psi,\eta) \in C^{2,\gamma}(\overline{D_\eta}) \times  C^{2,\gamma}(\R)$ of \eqref{sys:stream} is subject to 
	\begin{equation} \label{thm:slope:bounds1}
		\psi_y(x_2,\eta(x_2)) - \psi_y(x_1,\eta(x_1)) \leq C_1 (\eta(x_2)-\eta(x_1))
	\end{equation}
	for all $x_1 < x_2$ such that $\eta' \geq 0$ on $(x_1,x_2)$. If additionally $\omega \geq 0$ on $[0,1]$, then there exists $C_2>0$ such that
\begin{equation} \label{thm:slope:bounds2}
	\psi_y(x,\eta(x)) \geq \begin{cases}
	& C_2 (r-\eta(x)), \ \ r-\max\eta(x) \leq \tfrac12, \\
	& 1, \ \ r-\max\eta(x) > \tfrac12 
	\end{cases} 
\end{equation}
for all $x \in \R$; furthermore,
\begin{equation} \label{thm:slope:bounds3}
 |\eta'|^2 \leq C_2 (r-\eta)^{-1}, \ \ x \in \R.
\end{equation}
Constants $C_1$ and $C_2$ depend only on $r$ and $\omega$. 
\end{proposition}

This theorem shows that no wave breaking is possible for unidirectional Stokes waves with non-negative vorticity before stagnation occurs at crests.

\begin{proof} First we prove \eqref{thm:slope:bounds2}. For a given solution $(\psi,\eta)$ satisfying assumptions of the theorem we put
	\[
		f(x,y)=-k \psi+\Omega(\psi)+\tfrac{1}{2} \left(\psi_y^2+\psi_x^2\right)+y, \ \ \ (x,y) \in D_\eta,
	\]
	where $k > 0$ is a constant that will be specified later. A straightforward computation shows that
	\begin{equation}\label{thm:slope:eq1}
		{\mathcal L}f:=\Delta f + a f_x + b f_y = \left( 2 + (\omega(\psi)-2 k)(2\psi_y - k |\nabla \psi|^2)  \right) |\nabla \psi|^{-2},
	\end{equation}
	where
	\[
		\begin{split}
		  a  & = -2 (f_x - \psi_y  (\omega(\psi)-2 k))|\nabla \psi|^{-2}, \\
		  b  & = -2 (-2 + f_y - \psi_y  (\omega(\psi)-2 k))|\nabla \psi|^{-2}.
		\end{split}
	\]
	Note that the right-hand side in \eqref{thm:slope:eq1} is positive, provided $\omega \geq 0$ and $k$ is sufficiently small. Indeed, we require that
	\[
		2 \psi_y - k |\nabla \psi|^2 \geq  0, \ \ 2k (2 \psi_y - k |\nabla \psi|^2) \leq 2.
	\]
	The left inequality is justified when $k \leq 2/\sqrt{2r}$ (we used \eqref{sys:stream:bern} here), while for the second it is enough to assume that $k \leq 1/(2\sqrt{2r})$. Thus, if we choose $k$ as above, then the maximum principle shows that $f$ attains its global maximum at the boundary. On the bottom
	\[
	f_y = 1- k \psi_y > 0,
	\]
	provided $k$ is small enough in view of Proposition \ref{p:psiy:bot} (since $\psi_y$ is bounded at the bottom by a constant depending only on $r$ and $\omega$). On the upper boundary $f = k + \Omega(1)+r$ is constant and the Hopf lemma then gives $f_y(x,\eta(x)) \geq 0$ for all $x \in \R$. Computing $f_y$ explicitly and using the inequality for $f_y$, we obtain
	\begin{equation}\label{thm:slope:eq2}
		f_y = \frac{\eta'' \psi_y^2-k \left(\eta'^2+1\right) \psi_y+1}{\eta'^2+1} \geq 0, \ \ y = \eta(x).
	\end{equation}
	Note that
	\[
	\partial_x (\psi_y(x,\eta(x))) = - \eta'(x)\frac{\eta'' \psi_y^2+1}{(\eta'^2+1)\psi_y}
	\]
	and then \eqref{thm:slope:eq2} yields
	\[
	\partial_x (\psi_y(x,\eta(x))) + k \eta'(x) \leq 0
	\]
	everywhere so far $\eta'(x) \geq 0$. Now let $x \in (x_t,x_c)$, where $x_t$ and $x_c$ are coordinates of neighbour trough and crest respectively. Then $\eta'(t) \geq 0$ for all $t \in [x,x_c]$ and after integration we find
	\[
		\psi_y(x,\eta(x)) \geq k (\eta(x_c)-\eta(x)) + \psi_y(x_c,\eta(x_c)) \geq  k (\eta(x_c)-\eta(x))  + \sqrt{2(r-\eta(x_c))}.
	\]
	Now \eqref{thm:slope:bounds2} easily follows. The corresponding relation for $\eta'(x)$ can be obtained directly from the Bernoulli equation.
	
	Our argument for proving \eqref{thm:slope:bounds1} is similar. For some $k > 0$ we consider a function 
	\[
		g(x,y)=k \psi+\Omega(\psi)+\frac{1}{2} \left(\psi_y^2+\psi_x^2\right)+y, \ \ \ (x,y) \in D_\eta.
	\]
	It differs from $f$ in the sign of $k$. Then as before we obtain
	\[
		{\mathcal L}g:=\Delta f + a f_x + b f_y = \left( 2 + (\omega(\psi)+2 k)(2\psi_y + k |\nabla \psi|^2)  \right) |\nabla \psi|^{-2}.
	\]	
	We see that if $k > 0$ is large enough, then $\omega(\psi)+2 k \geq 0$ and ${\mathcal L}g > 0$ in $D_\eta$, while $g_y$ is always positive on $y=0$. Just as before we obtain an equality
	\[
		\partial_x (\psi_y(x,\eta(x))) - k \eta'(x) \leq 0,
	\]
	which after integration gives \eqref{thm:slope:bounds1}.
\end{proof}

\subsection{A lower bound for the wavelength}

Below we will prove that Stokes waves with a fixed Bernoulli constant can not have too small wavelengths.

\begin{theorem} \label{thm:wavelength} Assume that $\omega \in C^{1,\gamma}([0,1])$ and let $r>R_c$ be given.  Then there exists a positive constant $\Lambda_\star$ such that any Stokes wave solution to \eqref{sys:stream} has period greater than $\Lambda_\star$, while $\Lambda_\star$ depends only on $\omega$ and $r$.
\end{theorem}

Note that the main difficulty here is that we do not require any uniform bound for the slope $|\eta'|$ (uniform with respect to different solutions), which is a common assumption for many papers. Before giving a proof of the theorem we need to make some preparations. Given a Stokes wave solution $(\psi,\eta)$, even in $x$-variable and having a crest at the vertical line $x=0$, we define a function
\begin{equation} \label{function:G}
	G(x,y) = \int_{0}^x \left\{ \tfrac12 ((\psi_x)^2(x',y)-(\psi_y)^2(x',y))-\Omega(\psi(x',y))+\Omega(1) + r - \eta(x') \right\} \ dx'.
\end{equation}
This function similar to the flow force function (compare with the definition of function $F$ below) but the integration is taken with respect to the horizontal variable instead. The function $G$ is well defined inside the set
\[
D_\eta^0 = \{ (x,y) \in D_\eta: \ 0 \leq x \leq x_t \},
\]
where $x_t > 0$ is the smallest positive coordinate of a trough (function $\eta$ is strictly decreasing on the interval $(0,x_t)$). A straightforward computation gives
\begin{equation}\label{thm:wavelength:grad}
	G_x = \tfrac12 ((\psi_x)^2-(\psi_y)^2)-\Omega(\psi)+\Omega(1) + r - \eta, \ \ G_y = \psi_x \psi_y. 
\end{equation}
Thus, for $y = \eta(x)$ and $0<x<x_t$ we have
\begin{equation}\label{thm:wavelength:grad1}
G_x + \eta' G_y = -\tfrac12 ((\psi_x)^2+(\psi_y)^2)- \eta + r  =  0.
\end{equation}
This shows that $G$ is constant along the upper boundary of $D_\eta^0$ and because $G(0,\eta(0)) = 0$ by the definition we see that the constant is zero. Now it follows from \eqref{thm:wavelength:grad} that $G$ is also zero on the vertical sides of $D_\eta^0$. Thus, we can periodically extend $G$ into the whole region $D_\eta$ as an odd function. 

A useful property of $G$ is the following. Assume that a function $\psi$ just solves 
\[
\Delta \psi  +\omega(\psi) = 0 \ \ \text{in} \ \ D_\eta
\]
and is constant on $y = \eta$. Then we can define the function $G$ as above, which satisfies \eqref{thm:wavelength:grad}. Now we can claim that the Bernoulli equation for $\psi$ is equivalent to $G = const$ on $y=\eta$ as follows from \eqref{thm:wavelength:grad1}. This can be very useful when study weak limits of solutions to \eqref{sys:stream}. We will use this idea for proving Theorem \ref{thm:wavelength}.

Beside function $G$ we will also use some facts about the flow force function
\[
F(x,y) = \int_0^y \left\{ \tfrac12(\psi_y^2 - \psi_x^2) + \Omega(\psi)-\Omega(1) + r - y' \right\} dy'.
\]
This definition is inspired by the formula for the flow force constant, which is the boundary value of $F$:
\[
F = \FF \ \ \text{on} \ \ y = \eta.
\]
We note that 
\[
F_x = \psi_x \psi_y, \ \ F_y = \tfrac12(\psi_y^2 - \psi_x^2) + \Omega(\psi)-\Omega(1) + r - y.
\]
In particular gradients of $F$ and $G$ are bounded by a constant depending only on $\omega$ and $r$. Now we are ready to give a proof of Theorem \ref{thm:wavelength}.

\begin{proof}[Proof of Theorem \ref{thm:wavelength}]
	Assume that the claim is wrong and there exists a sequence of solutions $(\psi^{(j)},\eta^{(j)})$ whose periods tend to zero. We assume that all solutions are even $x$-variable and have a crest at the vertical line $x=0$. Let us put
	\[
		d = \liminf_{j \to +\infty} \min_\R \eta^{(j)}.
	\]
	Note that $d \geq d_-(r) > 0$; see \cite{Kozlov2015}. Without loss of generality we can assume that $\psi^{(j)}$ is convergent in $C^{\gamma}(\R \times [0,d])$ and $\psi$ is the limit function (we use compactness here, since gradients of $\psi^{(j)}$ are uniformly bounded). Furthermore, we can assume that functions $\psi^{(j)}$ are convergent in $C^{2,\gamma}(\R \times [0,d'])$ for any $0 < d' < d$. This is true because the strip $\R \times [0,d']$ is separated from the upper boundary of $D_{\eta^{(j)}}$ uniformly in $j$ and then the classical regularity theory (see \cite[Theorem 6.19]{GilbargTrudinger01}) gives
	\[
		\| \psi^{(j)}\|_{C^{2,\gamma}(\R \times [0,d'])} \leq C(r,\omega, d'), \ \ j \geq 1.
	\]
	Now it is left to use a compactness argument. 
	
	Let us point out some properties of the limiting function $\psi$. The definition of $d$ and the boundary relation \eqref{sys:stream:kintop} imply that $\psi = 1$ on $y = d$, while $\psi$ is independent of $x$-variable (since periods tend to zero by the assumption). Furthermore, the convergence in spaces $C^{2,\gamma}(\R \times [0,d']), \ d'\in (0,d)$ ensures that the limiting function solves the same elliptic equation:
	\[
		\Delta \psi + \omega(\psi) = 0 \ \ \text{for} \ \ 0 <y < d.
	\]
	In particular, $\psi$ is continuously differentiable up to the upper boundary $y = d$ (again by \cite[Theorem 6.19]{GilbargTrudinger01}). Let us show that $\psi$ is subject to the Bernoulli equation there with the same Bernoulli constant $r$. To show that we consider functions $G^{(j)}$ and $G$ defined by formula \eqref{function:G} for functions $(\psi^{(j)},\eta^{(j)})$ and $(\psi,\eta)$ respectively. Because gradients of $G^{(j)}$ are uniformly bounded we can assume that $G^{(j)}$ are convergent in $C^{\gamma}(\R \times [0,d])$. Now using the fact that $\psi^{(j)}$ converge to $\psi$ in every space $C^{2,\gamma}(\R \times [0,d'])$ with $d'\in (0,d)$ we conclude that the limit of $G^{(j)}$ must be $G$. On the other hand each $G^{(j)}$ is zero on $y = \eta^{(j)}$, while periods tend to zero. This shows that $G^{(j)}(x,d)$ tends to zero as $j \to +\infty$ for all $x$, so that $G = 0$ on $y = d$, which gives the Bernoulli equation for $\psi$ on $y = d$ as follows from \eqref{thm:wavelength:grad1}. Therefore, we either have $d = d_-(r)$ or $d = d_+(r)$. In particular, $d < r$, which will be important later.
	
	Let us prove now that 
	\[
	\limsup_{j \to +\infty} \max_\R \eta^{(j)} = d.
	\]
	Assume it is not true and (without loss of generality) $\eta^{(j)}(0) \to \hat{d} > d$ as $j \to +\infty$. Now we will consider flow force functions $F$ and $F^{(j)}$ at $x=0$, defined for $(\psi,\eta)$ and $(\psi^{(j)},\eta^{(j)})$ respectively. Let $\FF$ and $\FF_j$ be the corresponding flow force constants. Then clearly $\FF_j \to \FF$ (after passing to a subsequence) as $j \to +\infty$. This happens because gradients of $F^{(j)}$ are uniformly bounded and we can pass to a subsequence that converges in $C^{\gamma}(\R \times [0,d])$. Now because wavelengths tend to zero we obtain $|F^{(j)}(0,d) - \FF_j| \to 0$, which implies $\FF_j \to \FF$. On the other hand, we also have $F^{(j)}(0,\eta^{(j)}(0)) = \FF_j \to \FF$, which shows that
	\[
		\lim_{j \to +\infty} \int_d^{\eta^{(j)}(0)}   \left\{ \tfrac12 (\psi_y^{(j)}(0,y'))^2  + \Omega(\psi^{(j)}(0,y'))-\Omega(1) + r - y' \right\} dy'  = 0.
	\]
	Again, because wavelengths tend to zero, we have $\psi^{(j)}(0,y') \to 1$ as $j \to +\infty$ uniformly in $ y' \in [d,\eta^{(j)}(0)]$. Thus, we obtain 
	\[
		\lim_{j \to +\infty} \int_d^{\eta^{(j)}(0)} (r - y') dy' = \hat{d} (r - \tfrac12 \hat{d}) - d (r - \tfrac12 d) = (\hat{d}-d)(r - \tfrac12(d + \hat{d})) = 0.
	\]
	Now because $d < r$ and $\hat{d} \leq r$ we conclude $\hat{d} = d$, leading the a contradiction. This shows that amplitudes of $\eta_j$ tend to zero. But then inequality \eqref{thm:slope:bounds1} from Proposition \ref{thm:slope} shows that vertical derivatives $(\psi^{(j)})_y$ are uniformly separated from zero, so that $\max|\eta^{(j)}_x|$ are uniformly bounded. Therefore, $(\psi^{(j)},\eta^{(j)})$ is a smooth (by Proposition \ref{p:delta}) small-amplitude solution, whose wavelength must be close to $2\pi \tau^{-1} > 0$, where $\tau$ is the root of the dispersion equation (see \cite{Kozlov2017a, Groves2008, Kozlov2014} for details). This leads to a contradiction, because we assumed that wavelengths of $(\psi^{(j)},\eta^{(j)})$ tend to zero.
	
\end{proof}
\subsection{Global bifurcation theory}

		In this section we will formulate Theorem \ref{th:global} taken from \cite{BuffoniToland03}, which is the main tool of proving Theorem \ref{thm:vort}. Let 
\[
{\mathcal F}(x,\lambda) = 0, \ \ {\mathcal F}:U \subset X \times \R \to Y,
\]
be an analytic function between Banach spaces $X \times \R$ and $Y$, defined on an open subset $U \subset X \times \R$. Let $U_j$ be an increasing family of bounded open sets in $X \times \R$ such that $U = \cup_{j} U_j$ and $\overline{U_j} \subseteq U$ for all $j \in \mathbb{N}$.  Furthermore, we assume

\begin{itemize}
	\item[(A1)] $F(0,\lambda) = 0$ for all $\lambda \in \R$. In particular, $\{0\} \times \R \subset U$.
	\item[(A2)] for any $(x,\lambda) \in {\mathcal C}$ the linear operator $\partial_x F(x,\lambda):X \to Y$ is a Fredholm operator of index zero, where
	\[
	{\mathcal C} = \{ (x,\lambda) \in U: \ \ F(x,\lambda) = 0 \}.
	\]
	\item[(A3)] There are $\lambda_0 \in \R$ and $x_0 \in X$, $x_0 \neq 0$ such that
	\[
	\textrm{ker} (\partial_x F(0,\lambda_0)) = \{ t x_0: \ \ t\in \R\}
	\]
	and
	\[
	\partial^2_{x,\lambda}F(0,\lambda_0)[x_0,1] \notin \textrm{Im}(\partial_x F(0,\lambda_0)).
	\]
\end{itemize}
Assumptions (A1)-(A3) guarantee the existence of a smooth local curve of solutions
\[
{\mathcal C}_\dagger = \{ (x_\dagger(t),\lambda_\dagger(t)) \in {\mathcal C}, \ \ 0 \leq t < \epsilon \}
\]
for some small $\epsilon > 0$. Furthermore, we have
\[
\partial_t x_\dagger(0) = x_0.
\]
In order to extend ${\mathcal C}_\dagger$, we need the following additional properties.
\begin{itemize}
	\item[(B1)] All closed and bounded subsets of $\mathcal C$ are compact in $X \times \R$.
\end{itemize}

We say that a closed set ${\mathcal K}$ is a positive cone in $X$, if $tx \in X$ for all $t \geq 0$, provided $x\in {\mathcal K}$. Finally, we formulate three more conditions.
\begin{itemize}
	\item[(C1)] There is a positive cone ${\mathcal K} \subset X$ such that ${\mathcal C}_\dagger \subset {\mathcal K}$.
	\item[(C2)] The set ${\mathcal C}_0 \cap ({\mathcal K} \times \R)$ is open in ${\mathcal C}$, where
	\[
	{\mathcal C}_0 = \{(x,\lambda) \in {\mathcal C}: \ \ x \neq 0 \}.
	\]
	\item[(C3)] If $\xi \in \textrm{ker} (\partial_x F(0,\lambda)) \cap {\mathcal K}$ for some $\lambda \in \R$, then $\xi = t x_0$ with $t \geq 0$ and $\lambda = \lambda_0$.
\end{itemize}		

\begin{theorem}[Global bifurcation] \label{th:global}
	Assuming all conditions (A1)--(C3) hold true, there is a smooth curve 
	\[
	{\mathcal C}_\star = \{ (x_\star(t),\lambda_\star(t)) \in {\mathcal C}: \ \ t \in [0,+\infty) \}
	\]
	extending ${\mathcal C}_\dagger$ such that for any $j \in \mathbb{N}$ there is $t_j>0$ such that $(x_\star(t_j),\lambda_\star(t_j)) \notin U_j$.	Furthermore, ${\mathcal C}_\star$ admits a local analytic re-parametrization around every $t>0$, that is for any $t_0 > 0$ there is a continuous and injective function $\rho:[-1,1] \to \R$ such that $\rho(0) = t_0$ for which function $t \mapsto (x_\star(\rho(t)),\lambda_\star(\rho(t)))$ is analytic on $(-1,1)$.
\end{theorem}

Our formulation is slightly different to Theorem \ref{th:global} from \cite{BuffoniToland03} and is very close to Theorem 6 in \cite{Constantin2016} (see a discussion therein).

\section{Proof of Theorem \ref{thm:vort} and Corollary \ref{cor:vort}}

In order to apply Theorem \ref{th:global} we need to reformulate our problem in a suitable way. Note that for any $r \in (R_c,d_0)$ there is a subcritial stream solution $H(p;s_-(r))$ corresponding to the same Bernoulli constant $r$. Thus, we put
\[
w = h-H
\]
and using \eqref{sys:hs} obtain the corresponding problem for $w$, which is
\begin{subequations}\label{sys:ws}
	\begin{alignat}{2}
	\label{sys:ws:main}
	- \left(\frac{w_p}{H_p^3}\right)_p-\lambda^2 \frac{w_{qq}}{H_p} + N_1(w,\lambda)&=0 &\qquad& \text{in } S,\\
	\label{sys:ws:top}
	w_p - H_p^{3} w + N_2(w,\lambda)  &= 0 &\quad& \text{on }p=1,\\
	\label{sys:ws:bot} 
	w &= 0 &\quad& \text{on }p=0,
	\end{alignat}
\end{subequations}
where
\[
\begin{split}
&  N_1(w,\lambda) = \lambda^2 \left\{ \left(\frac{w_q w_p}{h_p H_p}\right)_q + \left(\frac{w_q^2}{2h_p^2}\right)_p  \right\} + \left(\frac{w_p^2(2h_p+H_p)}{2H_p^3h_p^2}\right)_p, \\
& N_2(w,\lambda) = - \frac{w_p^2(2h_p+H_p)}{2h_p^2} - \lambda^2 \frac{w_q^2 H_p^3}{2 h_p^2}.
\end{split}
\]
Let $X$ be the subspace of $C^{2,\gamma}(\overline{S})$ which consists of even and $2\pi$-periodic in $q$-variable functions $w$ such that $w(q,0) = 0$ for $q \in \R$. We also define the range spaces $Y_1 = C^{0,\gamma}_{per}(\overline{S})$ and $Y_2 = C^{1,\gamma}_{per}(\R)$. Next we put
\[
\begin{split}
& U = \{ w \in X : 0 < w_p + H_p < + \infty \ \ \text{in} \ \ \overline{S} \} \cup \{ \lambda \in \R: \ \lambda > 0 \}, \ \ \\
& U_j = \{  w \in X : 0 < w_p + H_p < j \ \ \text{in} \ \ \overline{S} \ \ \text{and} \ \ \|w\|<j \} \cup \{  \lambda \in \R: \frac1j <  \lambda  < j   \}, \ j\in\mathbb{N}.
\end{split}
\]
Now we define
\[
{\mathcal F}(w,\lambda) = \left(- \left(\frac{w_p}{H_p^3}\right)_p-\lambda^2 \frac{w_{qq}}{H_p} + N_1(w,\lambda),w_p - H_p^{3} w + N_2(w,\lambda) \right),
\]
so that ${\mathcal F}:U \to Y := Y_1 \times Y_2$ is analytic. Finally, we define a positive cone as
\[
{\mathcal K} = \{w \in X : \ w(0,1) > 0 \}.
\]
\subsection{Dispersion equation}

The kernel of $D_{w}{\mathcal F}(0,\lambda)$ consists of all $w \in X$ such that
\begin{subequations}\label{sys:ker}
	\begin{alignat}{2}
	\label{sys:ker:main}
	- \left(\frac{w_p}{H_p^3}\right)_p-\lambda^2 \frac{w_{qq}}{H_p} &=0 &\qquad& \text{in } S,\\
	\label{sys:ker:top}
	w_p - H_p^{3} w  &= 0 &\quad& \text{on }p=1,\\
	\label{sys:ker:bot} 
	w &= 0 &\quad& \text{on }p=0.
	\end{alignat}
\end{subequations}
Separating variables we find that this problem has a non-zero solution if and only if there is a function $\varphi \neq 0$  such that
\begin{subequations}\label{sys:ker}
	\begin{alignat}{2}
	\label{sys:disp:main}
	- \left(\frac{\varphi_p}{H_p^3}\right)_p = -\lambda^2 \frac{\varphi }{H_p} &=0 &\qquad& \text{on }  \ (0,1),\\
	\label{sys:disp:top}
	\varphi_p(1) - H_p^{3}(1) \varphi(1)  &= 0, &\quad&\\
	\label{sys:disp:bot} 
	\varphi(0) &= 0. &\quad&
	\end{alignat}
\end{subequations}
This Sturm-Liouville problem is well known (see \cite{Kozlov2014a, Kozlov2017a, Kozlov2014}) and appears as a dispersion equation for steady water waves. For subcritical stream solutions $H = H(p;s_-(r))$ there is a unique $\lambda_0 > 0$ such that the problem above has a non-trivial eigenfunction $\varphi_0(p)$. Thus, the kernel of operator $D_{w}{\mathcal F}(0,\lambda)$ is trivial for $\lambda \neq \lambda_0$ and coincides with
\[
\{ t \varphi_0(p) \cos(q), \ \ t \in \R \}
\]
for $\lambda = \lambda_0$. Note that since we only consider even solutions the kernel is at most one-dimensional.

\subsection{Application of Theorem \ref{th:global} and proof of Theorem \ref{thm:vort}}

Now we are ready to apply Theorem \ref{th:global}. The first property (A1) follows immediately from the definition of ${\mathcal F}$. For the validity of (A2) we refer to \cite[Lemma 4.3]{ConstantinStrauss04}, while (A3) follows from \cite{Kozlov2014}. Furthermore, (C1)-(C3) follow from definitions. Thus we only need to prove (B1). Let $K$ be a bounded and closed in $X$ subset of $U$. Then 
\[
\sup_{(w,\lambda) \in K} \sup_S H_p+w_p = \delta^{-1} < + \infty, \ \ \sup_{(w,\lambda) \in K} \lambda + \frac{1}{\lambda} <  +\infty, \ \ \sup_{(w,\lambda) \in K} \sup_S |w_q| = M  < + \infty.
\]
For a given $(w,\lambda) \in K$ let $\psi$ be the corresponding to $h = H + w$ stream function in original variables and $\eta$ is the surface profile. Then $\inf_{D_\eta} \psi_y \geq \delta > 0$ uniformly in $K$, while $\max |\eta'| \leq M$. Then Proposition \ref{prop:psibounds} shows that norms
\[
\|\psi\|_{C^{3,\alpha}(\overline{D_\eta})},  \|\eta\|_{C^{3,\alpha}(\R)} \leq C(r,\delta)
\] 
are bounded uniformly in $K$. Thus, norms $\|w\|_{C^{3,\alpha}(\overline{S})}$ are also uniformly bounded which gives the compactness.

Now application of Theorem \ref{th:global} gives a curve of solutions ${\mathcal C}_\star$ and a sequence $t_j \to +\infty$ such that ${\mathcal C}_\star(t_j) = (w_j, \lambda_j)$ lies outside $U_j$. Just as in \cite[Lemma 5.1,Lemma 5.2]{ConstantinStrauss04} we show that every solution ${\mathcal C}_\star(t)$ represents a Stokes wave, even and monotone between each neighbour crest and trough. Furthermore, the definition of $U_j$ implies that one of the following options takes place:
\begin{itemize}
	\item[(1)] $\|w_j\|_{C^{2,\gamma}(\overline{S})} \to \infty$ as $j \to +\infty$;
	\item[(2)] $\sup_{S} (H_p + (w_j)_p)\to \infty$ as $j \to +\infty$;
	\item[(3)] $\lambda_j \to 0$ as $j \to +\infty$;
	\item[(4)] $\lambda_j \to +\infty$ as $j \to +\infty$;	
\end{itemize}
Let us check claims (I)-(IV) of the theorem. Assume that $(IV)$ does not hold (no wave breaking). Then in view of Proposition \ref{p:delta} options (1) and (2) are equivalent, that is (1) holds true if and only if (2) is true. Thus, (1) or (2) gives (II), while (III) follows from (3). The remaining option (4) is forbidden by Theorem \ref{thm:wavelength}.

\begin{proof}[Proof of Corollary \ref{cor:vort}] Let $(\psi^{(j)},\eta^{(j)})$ be a family of Stokes waves provided by Theorem \ref{thm:vort} and corresponding to the Bernoulli constant $r$. Let us show that there exists a subsequence, converging to a weak solution $(\psi,\eta)$ of the problem \eqref{sys:stream}, which is regular everywhere outside the stagnation points, where $\eta = r$. This can be done as follows. For a given $\epsilon > 0$ we define truncated functions
	\[
		\eta^{(j)}_\epsilon(x) = \begin{cases}
		& \eta^{(j)}(x), \ \ \eta^{(j)}(x) \leq r-\epsilon, \\
		& r-\epsilon, \ \ \eta^{(j)}(x) > r-\epsilon.
		\end{cases}
	\]
Let us also consider the corresponding  domains $D^{(j)}_\epsilon = D_{\eta^{(j)}_\epsilon}$. The interior and boundary Harnack inequalities (see \cite[Theorem 8.20, Theorem 8.26]{GilbargTrudinger01}) guarantee that
\[
	\inf_{D^{(j)}_\epsilon} \psi^{(j)}_y \geq C(\epsilon) > 0, 
\]
where constant $C(\epsilon)$ depends only on $\epsilon, r, \omega$ and is independent of $j$. Now application of Proposition \ref{p:delta} gives
\[
	\|\psi^{(j)}\|_{C^{3,\gamma}(D^{(j)}_\epsilon)} \leq C(\epsilon).	
\]
Thus, we can find a subsequence $(\psi^{(j_k)},\eta^{(j_k)})$ for which $\{\eta^{(j_k)}_\epsilon\}$ is convergent in $C^{\gamma}(\R)$ to a Lipschitz function $\eta_\epsilon$, while stream functions $\psi^{(j_k)}$ converge to some $\psi_\epsilon$ in every space $C^{2,\gamma}(K)$ for all compact subsets of $D_{\eta_\epsilon} \cup \R \times \{0\}$. We repeat this argument for a sequence $\epsilon_k \to 0$, every time passing to a subsequence, and obtain a pair $(\psi,\eta)$, where $\eta$ is smooth everywhere expect, possibly, points where $\eta = r$ (single crests). The corresponding stream function has a bounded gradient inside $D_\eta$, solves \eqref{sys:stream:lap} in $D_\eta$, and is also smooth in $\overline{D_\eta} \cap \{y < r\}$. The boundary relations \eqref{sys:stream:kintop} and \eqref{sys:stream:bern} are satisfied everywhere along the boundary except points, where $\eta = r$. Thus, $(\psi,\eta)$ is a weak solution in the sense of \cite{Varvaruca09}, which is an extreme Stokes wave if option (I) takes place, regular solitary wave in case of (II) or an extreme solitary wave in (III). Note that option (IV) implies (I) in view of Proposition \ref{thm:slope}. Now it is left to apply \cite[Theorem 5.2]{Varvaruca09}. The Lipschitz regularity of the surface profile around stagnation points follows from \cite[Theorem 5.2]{Varvaruca09}, while outside stagnation points it is guaranteed by Proposition \ref{thm:slope}. In a similar way one proves $C^{1}-$regularity for the stream function along the boundary. To ensure global regularity (in particular, continuity of $\psi_y$ at stagnation points) it is enough to use the Boundary Harnack principle for $\psi_y$ from \cite[Theorem 8.25]{GilbargTrudinger01}. The continuity for $\psi_x$ now follows from the relation $\psi_x = -\eta' \psi_y$, valid along the top boundary.

Finally, no solitary waves exist for $F < 2$ as follows from \cite{Wheeler15b} and no solitary waves exist for $r \geq d_0$, which gives the desired value of $r_\star$. This finishes our proof of Corollary \ref{cor:vort}.

\end{proof}

\begin{proof}[Proof of Corollary \ref{cor:irr}] The statement follows from Corollary \ref{cor:vort} taking into account results of \cite{Varvaruca2011} and the fact that no solitary waves exist for $F \geq 2$ (see \cite{starr}), where $F$ is the Froude number. The latter bound for $F$ is equivalent to $r \geq 2^{2/3}$.
\end{proof}

\section{Proof of Theorem \ref{thm:reg}}

This section is entirely dedicated to the irrotational case $\omega = 0$.

\subsection{Flow force function formulation}

It well known that the water wave problem admits another constant of motion (see \cite{BENJAMIN1984} for more details), the flow force, defined as 
\begin{equation} \label{flowforce}
\FF = \int_0^{\eta}(\tfrac12(\psi_y^2 - \psi_x^2) - y + r )\, dy.
\end{equation}	
The latter expression is independent of $x$, which makes it of special importance for spatial dynamics (see \cite{Groves2008} and references therein). Beside, it plays an important role in classification of steady waves; see \cite{Benjamin95, Kozlov2017a, Lokharu2020a}.

Based on the definition of the flow force constant $\FF$, we introduce the corresponding flow force function
\begin{equation} \label{fff}
F = \int_0^{y}(\tfrac12(\psi_y^2(x,y') - \psi_x^2(x,y')) - y' + r )\, dy'.
\end{equation}
Just as in \cite{Basu2020} we can reformulate the water wave problem in terms of the function $F$. It is straightforward to obtain
\begin{equation} \label{fff:grad}
F_x = \psi_x \psi_y, \ \ F_y = \tfrac12(\psi_y^2 - \psi_x^2) - y + r.
\end{equation}
Thus, we arrive to an equivalent formulation given by
\begin{subequations}\label{sys:fff}
	\begin{alignat}{2}
	\label{sys:fff:lap}
	\Delta F + 1&=0 &\qquad& \text{for } 0 < y < \eta,\\
	\label{sys:fff:bern}
	\tfrac 12 (-\eta' F_x + F_y) +  y  &= r &\quad& \text{on }y=\eta,\\
	\label{sys:fff:kintop} 
	F  &= \FF &\quad& \text{on }y=\eta,\\
	\label{sys:fff:kinbot} 
	F  &= 0 &\quad& \text{on }y=0.
	\end{alignat}
\end{subequations}

This new formulation in terms of flow force function $F$ is of special interest as shown recently in \cite{Lokharu2020a}. However, in the present paper we only need $F$ as a comparison function for proving the next result.

\begin{proposition} \label{prop:psibounds}
	For any $r>r_c$ and any stream function $\psi \in C^{2,\gamma}(\overline{D_\eta})$ solving \eqref{sys:stream} we have
	\begin{equation}\label{psiy:bot}
		\frac{1}{r} < \psi_y(x,0) < \frac{4}{r}
	\end{equation}	
	for all $x\in \R$, where $r$ is the corresponding Bernoulli constant.
\end{proposition}
\begin{proof}
	The left inequality can be obtained by comparing $(\psi,\eta)$ to the stream solution $U(y) = \frac1r y$, whose depth is $r$. Indeed, let $h$ and $H$ be the height functions of $\psi$ and $U$ respectively. Then function $w = h-H$ is negative for $p=1$ by the choice of $U$, while $w = 0$ on $p=0$. Thus, $w \leq 0$ by the maximum principle and therefore $w_p(q,0) < 0$ for all $q \in \R$ by the Hopf lemma. The latter is equivalent to $\psi_y(0,y) > U_y(0,y) = \frac1r$.
	
	In order to obtain the upper bound we need to consider the function $f = F - \FF \psi$, where $F$ is the flow force function \eqref{fff} corresponding to $\psi$. It is clear that $\Delta f = -1$ and $f$ is zero on the boundary of $D_\eta$. Thus, it positive inside $D_\eta$ by the maximum principle and $f_y(x,0) > 0$ by the Hopf lemma. This gives
	\[
		\tfrac12\psi_y^2(x,0) + r \geq \FF \psi_y(x,0).
	\]
	Now using \eqref{sys:stream:bern} we conclude $\psi_y(x,0) \leq \FF^{-1} 2r$. It is left to use inequality $\FF > \tfrac12 r^2$ obtained in \cite{Lokharu2020a}.
\end{proof}

Proof of Theorem \ref{thm:reg} consists of several steps, including interior and boundary estimates, which are based on \eqref{psi:rminy}, which to be proved first. The upper bound in \eqref{psi:rminy} is well known. It can be obtain by applying the maximum principle to the super-harmonic function 
\[
g = \tfrac12(\psi_x^2 + \psi_y^2) + y.
\]
On the bottom $g_y = 1$, so the global maximum can not be attained there. On the surface $g = r$ by \eqref{sys:stream:bern}, so that $g < r$ in $D_\eta$, which gives $\tfrac12 \psi_y^2 < r-y$ as desired.

The bottom inequality in \eqref{psiy:bot} is an important step in proving \eqref{psi:reg} and seems to be new. Our argument involves a function
\[
f = \frac{F_y}{r-y}.
\]
A direct computation shows that
\[
(r-y)\Delta f - f_y = 0 \ \ \text{in} \ \ D_\eta.
\]
Thus, the global minimum of $f$ is attained at the boundary by the maximum principle (note that we consider sufficiently regular solutions, which forbids the presence of stagnation and so $r-y$ is separated from zero in $D_\eta$). On the surface
\[
f = \frac{2}{1 + \eta'^2} > \tfrac85,
\]
since $|\eta'| \leq 1/2$ for all Stokes waves; see \cite{Amick1987}. The constant $\tfrac85$ is not of any particular importance and we only need that it is greater than $1$. On the bottom
\[
f = \frac{\tfrac12 \psi_y^2 + r}{r} > 1 + \frac{1}{2r^3},
\]
where we used \eqref{psiy:bot}. Now since $r > \tfrac32$ we obtain $f > 1 + \frac{1}{2r^3}$ in $D_\eta$. Using the definition of $F$, we conclude
\[
F_y = \tfrac12 (\psi_y^2 - \psi_x^2) + r - y > \left(1 + \frac{1}{2r^3}\right) (r-y).
\]
This leads to \eqref{psi:rminy}, where $C_2 = (2r^3)^{-3}$.

Let us prove \eqref{psi:reg}. Let $(x_1,y_1)$ and $(x_2,y_2)$ be two points in $D_\eta$ and $\delta = \sqrt{(x_1-x_2)^2+(y_1-y_2)^2}$. If both $r-y_1$ and $r-y_2$ are less than $\delta$, then 
\[
|\psi_y(x_1,y_1)-\psi_y(x_2,y_2)| \leq |\psi_y(x_1,y_1)|+|\psi_y(x_2,y_2)|\leq|r-y_1|^{\tfrac12} + |r-y_2|^{\tfrac12} \leq 2 \delta^{\tfrac12}
\]
by \eqref{psi:rminy}. Thus, without loss of generality, we can assume that $r-y_2 \geq \delta$. Now if $r-y_1 < \tfrac14 C_2\delta$, then $\psi_y(x_2,y_2) > \psi_y(x_1,y_1)$ by \eqref{psi:rminy} and we have
\[
\begin{split}
\psi_y(x_2,y_2)-\psi_y(x_1,y_1) & \leq |r-y_2|^{\tfrac12} =  |r-y_2|^{\tfrac12} - |r-y_1|^{\tfrac12} +  |r-y_1|^{\tfrac12} \\
& \leq \frac{y_2 - y_1}{|r-y_2|^{\tfrac12}} + \tfrac12 \delta^{\tfrac12}\leq \tfrac32 \delta^{\tfrac12}.
\end{split}
\]
Therefore, we only need to consider the case $r-y_2 \geq \delta$ and $r-y_1 \geq \tfrac14 C_2\delta = C_3 \delta$. The boundary case is covered by the following lemma.

\begin{lemma} \label{l:reg:boundary}
	Let $y_1 = \eta(x_1)$ and $\delta_1:= r-y_1 > 0$. There exist $\alpha \in (0,1)$ and $C > 0$ such that 
	\[
	\|\psi\|_{C^{1,\alpha}(B\cap D_\eta)} \leq C,
	\]
	where $B$ is the ball of radius $\tfrac12 \delta_1$ at $(x_1,y_1)$. Both constants $\alpha$ and $C$ depend only on $r$ and are independent of $\delta_1$ and $(x_1,y_1)$.
\end{lemma}

\begin{proof}
	Note that within $B \cap D_\eta$ we have 
	\begin{equation} \label{l1:eq1}
	C^{-1} \delta_1 < \psi_y^2 \leq C \delta_1,
	\end{equation}
	where $C$ depends only on $r$. Let $h$ be the height function corresponding to $\psi$, which solves a homogenous elliptic problem
	\[
		\frac{1+h_q^2}{h_p^2} h_{pp} - 2 \frac{h_q}{h_p} h_{qp} + h_{qq} = 0,
	\]
	subject to the boundary condition \eqref{sys:h:top}. If $\delta_1$ is small, then $h_p$ is large and we can not use any classical regularity results directly. In order to avoid the singularity, we perform the following anisotropic scaling 
	\[
		\hat{q} = q-x_1, \ \ \hat{p} = \delta_1^{-\tfrac12} (1-p), 
	\]
	where $\hat{h}(\hat{q},\hat{p}) = h(q,p)$ is the new unknown function, solving
	\begin{equation}\label{eq:hath}
	\frac{1+\hat{h}_{\hat{q}}^2}{\hat{h}_{\hat{p}}^2} \hat{h}_{\hat{p}\hat{p}} - 2 \frac{\hat{h}_{\hat{q}}}{h_{\hat{p}}} \hat{h}_{\hat{q}\hat{p}} + \hat{h}_{\hat{q}\hat{q}} = 0.
	\end{equation}
	This problem shall be considered in a neighbourhood of $(0,0)$, which is specified below. 	
	Note that if $\omega \in B$, then $(\hat{q},\hat{p})$ corresponding to $(q,p) = (x,\psi\omega)$ belongs to the rectangle
	\[
		\hat{B} = \{ (\hat{q},\hat{p}): \	|\hat{q}| \leq C \delta_1, \ 0 \leq \hat{p} \leq C \delta_1 \}, 
	\]
	where $C$ depends only on $r$. Indeed, the inequality for $\hat{q}$ is trivial, while
	\[
	\hat{p} = \delta_1^{-\tfrac12} (1-p) \leq \delta_1^{-\tfrac12} \max_{y_1 \geq y' \leq \eta(x_1)} \psi(x_1,y') (\eta(x_1)-y_1) \leq C \delta_1,
	\]
	which follows from \eqref{l1:eq1} and because $\eta(x_1)-y_1 \leq r-y_1= \delta_1$. The main advantage of the scaling is that we have
	\[
	   C^{-1} \leq \hat{h}_{\hat{p}} \leq C \ \ \text{in} \ \ \hat{B}
	\]
	with some constant, depending only on $r$. Thus, the linear second order operator in \eqref{eq:hath} is uniformly elliptic in $\hat{B}$, while is subject to the boundary relation	
	\begin{equation}\label{eq:hath:bern}
	\frac{1 + \hat{h}_{\hat{q}}^2}{2} = \frac{r-h}{\delta_1} \hat{h}_{\hat{p}}^2 \ \ \text{on} \ \ \hat{p}=0,
	\end{equation}	
	which follows from the Bernoulli equation \eqref{sys:h:top} for $h$. Our aim is to prove uniform a H\"older regularity for $\hat{h}_{\hat{q}}$ and $\hat{h}_{\hat{p}}$ in $\tfrac23 \hat{B}$. First we prove this for $f = \hat{h}_{\hat{q}}$. For this purpose we differentiate \eqref{eq:hath} and \eqref{eq:hath:bern} with respect to $\hat{q}$, we obtain a system for $f$, given by
	\[
	\begin{split}
		& \frac{1+\hat{h}_{\hat{q}}^2}{\hat{h}_{\hat{p}}^2} f_{\hat{p}\hat{p}} - 2 \frac{\hat{h}_{\hat{q}}}{h_{\hat{p}}} f_{\hat{q}\hat{p}} + f_{\hat{q}\hat{q}} + a f_{\hat{q}}f_{\hat{p}} + b f_{\hat{p}}^2 = 0, \\
		& \delta_1 \hat{h}_{\hat{q}} f_{\hat{q}} + f \hat{h}_{\hat{p}}^2 - 2 \frac{r-h}{\delta_1} \delta_1 f_{\hat{p}} \ \ \text{on} \ \ \hat{p}=0,
	\end{split}		
	\]	
	where $a,b$ are $L^{\infty}$ in $\hat{B}$ (with norms depending only on $r$). Now we can scale $\hat{q} = \delta_1 \tilde{q}$ and $\hat{p} = \delta_1 \tilde{p}$, where $\tilde{f}(\tilde{q},\tilde{p}) = f(\hat{q},\hat{p})$. This leads to the problem
	\[
		\begin{split}
			& \frac{1+\hat{h}_{\hat{q}}^2}{\hat{h}_{\hat{p}}^2} \tilde{f}_{\tilde{p}\tilde{p}} - 2 \frac{\hat{h}_{\hat{q}}}{h_{\hat{p}}} \tilde{f}_{\tilde{q}\tilde{p}} + \tilde{f}_{\tilde{q}\tilde{q}} + a f_{\hat{q}}f_{\hat{p}} + b f_{\hat{p}}^2 = 0, \\
			& \hat{h}_{\hat{q}} \tilde{f}_{\tilde{q}} + \tilde{f} \hat{h}_{\hat{p}}^2 - 2 \frac{r-h}{\delta_1}  \tilde{f}_{\tilde{p}} \ \ \text{on} \ \ \tilde{p}=0,
		\end{split}		
	\]
	where all coefficients are bounded and separated from zero in 
	\[
	\tilde{B} = \{ (\tilde{q},\tilde{p}): \	|\tilde{q}| \leq C, \ 0 \leq \tilde{p} \leq C  \}.
	\]
	Now we can apply Theorem 2.4 in \cite{Lieberman1986}, which provides constants $\alpha \in (0,1)$ and $\tilde{C}>0$ depending only on $r$ such that
	\[
	\|\tilde{f}\|_{C^{\alpha}(\tfrac23\tilde{B})} \leq \tilde{C}.
	\]
	In particular, we arrive to the desired estimate for $f = \hat{h}_{\hat{q}}$:
	\begin{equation} \label{reg:f}
	\|f\|_{C^{\alpha}(\tfrac23 \hat{B})} \leq \tilde{C} \delta_1^{-\alpha}.
	\end{equation}
	A similar bound for $\hat{h}_{\hat{p}}$ at the boundary can be obtained from \eqref{eq:hath:bern}, leading to
	\begin{equation} \label{reg:hathp}
	\|\hat{h}_{\hat{p}}\|_{C^{\alpha}(\hat{\Gamma})} \leq \tilde{C} \delta_1^{- \alpha},
	\end{equation}
	where $\hat{\Gamma} = \tfrac23\hat{B} \cap \{\hat{p}=0\}$. Therefore, relations \eqref{reg:f} and \eqref{reg:hathp} lead to
	\[
		\Big\|\frac{\psi_x}{\psi_y}\Big\|_{C^{\alpha}(\Gamma)} \leq \tilde{C} \delta_1^{-\alpha} , \ \Big\|\frac{1}{\psi_y}\Big\|_{C^{\alpha}(\Gamma)} \leq \tilde{C} \delta_1^{-\alpha - \tfrac12},
	\]
	where $\Gamma = \{\tfrac23 B \cap \partial D_\eta\}$. Now the classical regularity theory (see Theorem 8.29 \cite{GilbargTrudinger01}) gives
	\[
		\Big\|\frac{\psi_x}{\psi_y}\Big\|_{C^{\alpha}(\tfrac12 B \cap D_\eta)} \leq \tilde{C} \delta_1^{-\alpha} , \ \Big\|\frac{1}{\psi_y}\Big\|_{C^{\alpha}(\tfrac12 B \cap D_\eta)} \leq \tilde{C} \delta_1^{-\alpha - \tfrac12},
	\]	
	Taking into account \eqref{l1:eq1}, we obtain
	\[
		\|\psi_x\|_{C^{\alpha}(\tfrac12 B \cap D_\eta)} , \ \|\psi_y\|_{C^{\alpha}(\tfrac12 B \cap D_\eta)} \leq \tilde{C} \delta_1^{-\alpha + \tfrac12} \leq C(r),	
	\]
	where $\alpha$ is assumed to be less than $\tfrac12$. This finishes the proof of Lemma \ref{l:reg:boundary}.
\end{proof}

An interior estimate for points that are close to the boundary is provided below.

\begin{lemma} \label{l:reg:int} Let $\delta_2:=\min{(\eta(x_1)-y_1,\eta(x_2)-y_2)}$ and $\delta < \tfrac14\delta_2$. Then for any $\theta \in (0,1)$ there is $C > 0$ such that 
	\[
	\|\psi\|_{C^{1,\theta}(B)} \leq C \delta_2^{-\tfrac12},
	\]
	where $B$ is the ball of radius $\tfrac12 \delta_2$ at $(x_1,y_1)$. The constant $C$ depends only on $r$ and $\theta$.
\end{lemma}
\begin{proof}
	Let $\chi$ be a smooth function such that $\chi=1$ on B, $\chi = 0$ outside $\tfrac23 B$ and 
	\[
	\|\nabla\chi\|_\infty \leq C \delta_2^{-1}, \ \ \|\Delta\chi\|_\infty \leq C \delta_2^{-2}.
	\]
	Then for $f = \chi (1-\psi)$ we have
	\[
	\Delta f = 2\nabla \chi \nabla \psi + \Delta \chi (1-\psi).
	\]
	Note that
	\[
	\|\nabla \chi \nabla \psi\|_{\infty} \leq C\delta_2^{-\tfrac12}, \ \ \|\Delta \chi (1-\psi)\|_\infty \leq C \delta_2^{-\tfrac12}.
	\]
	Therefore, for any $\theta \in (0,1)$ we have $\|f\|_{C^{1,\theta}(B)} \leq C \delta_2^{-\tfrac12}$ by Theorem 8.33 \cite{GilbargTrudinger01}, so that
	\[
		\|\psi\|_{C^{1,\theta}(B)} \leq C \delta_2^{-\tfrac12}
	\]	
	as desired.
\end{proof}

Now we can finish our proof of Theorem \ref{thm:reg}. Indeed, we are left to consider the case  $r-y_2 \geq \delta$ and $r-y_1 \geq \tfrac14 C_2\delta = C_3 \delta$. In view of Lemma \ref{l:reg:boundary} we can additionally assume that $\eta(x_1)-y_1$ and $\eta(x_2)-y_2$ are greater or comparable with $\delta$. Then it is left to apply Lemma \ref{l:reg:int} (splitting the segment connecting $(x_1,y_1)$ and $(x_2,y_2)$, if necessary), which yields
\[
|\nabla\psi(x_1,y_1)-\nabla\psi(x_2,y_2)| \leq C \delta_2^{-\tfrac12} \delta^{\theta} \leq C \delta^{\theta - \tfrac12}.
\]
It is left to choose $\theta = \alpha + \tfrac12$, where $\alpha$ is the constant from Lemma \ref{l:reg:boundary}.
\vspace{4mm}

\noindent {\bf Acknowledgements.} V.~K. was supported by the Swedish Research
Council (VR), 2017-03837.

\bibliographystyle{siam}
\bibliography{bibliography}

\begin{thebibliography}{10}

\bibitem{Allen2019}
{\sc M.~Allen and H.~Shahgholian}, {\em A new boundary harnack principle
  (equations with right hand side)}, Archive for Rational Mechanics and
  Analysis, 234 (2019), pp.~1413--1444.

\bibitem{Amick1987}
{\sc C.~J. Amick}, {\em Bounds for water waves}, Archive for Rational Mechanics
  and Analysis, 99 (1987), pp.~91--114.

\bibitem{Amick1987a}
{\sc C.~J. Amick and L.~E. Fraenkel}, {\em On the behavior near the crest of
  waves of extreme form}, Transactions of the American Mathematical Society,
  299 (1987), pp.~273--273.

\bibitem{AmickFraenkelToland82}
{\sc C.~J. Amick, L.~E. Fraenkel, and J.~F. Toland}, {\em On the {S}tokes
  conjecture for the wave of extreme form}, Acta Math., 148 (1982),
  pp.~193--214.

\bibitem{AmickToland81a}
{\sc C.~J. Amick and J.~F. Toland}, {\em On periodic water-waves and their
  convergence to solitary waves in the long-wave limit}, Philos. Trans. Roy.
  Soc. London Ser. A, 303 (1981), pp.~633--669.

\bibitem{AmickToland81b}
{\sc C.~J. Amick and J.~F. Toland}, {\em On solitary water-waves of finite
  amplitude}, Arch. Rational Mech. Anal., 76 (1981), pp.~9--95.

\bibitem{Baesens1992}
{\sc C.~Baesens and R.~S. Mackay}, {\em Uniformly travelling water waves from a
  dynamical systems viewpoint: some insights into bifurcations from stokes'
  family}, Journal of Fluid Mechanics, 241 (1992), pp.~333--347.

\bibitem{Basu2020}
{\sc B.~Basu}, {\em A flow force reformulation for steady irrotational water
  waves}, Journal of Differential Equations, 268 (2020), pp.~7417--7452.

\bibitem{BENJAMIN1984}
{\sc T.~B. Benjamin}, {\em Impulse, flow force and variational principles},
  {IMA} Journal of Applied Mathematics, 32 (1984), pp.~3--68.

\bibitem{Benjamin95}
{\sc T.~B. Benjamin}, {\em Verification of the {B}enjamin-{L}ighthill
  conjecture about steady water waves.}, J. Fluid Mech.\/, 295 (1995),
  pp.~337--356.

\bibitem{BuffoniDancerToland00a}
{\sc B.~Buffoni, E.~N. Dancer, and J.~F. Toland}, {\em The regularity and local
  bifurcation of steady periodic water waves}, Arch. Ration. Mech. Anal., 152
  (2000), pp.~207--240.

\bibitem{BuffoniToland03}
{\sc B.~Buffoni and J.~Toland}, {\em Analytic theory of global bifurcation},
  Princeton Series in Applied Mathematics, Princeton University Press,
  Princeton, NJ, 2003.
\newblock An introduction.

\bibitem{Constantin11b}
{\sc A.~Constantin}, {\em Nonlinear water waves with applications to
  wave-current interactions and tsunamis}, vol.~81 of CBMS-NSF Regional
  Conference Series in Applied Mathematics, Society for Industrial and Applied
  Mathematics (SIAM), Philadelphia, PA, 2011.

\bibitem{Constantin2007}
{\sc A.~Constantin, M.~Ehrnström, and E.~Wahl{\'{e}}n}, {\em Symmetry of
  steady periodic gravity water waves with vorticity}, Duke Mathematical
  Journal, 140 (2007), pp.~591--603.

\bibitem{ConstantinStrauss04}
{\sc A.~Constantin and W.~Strauss}, {\em Exact steady periodic water waves with
  vorticity}, Comm. Pure Appl. Math., 57 (2004), pp.~481--527.

\bibitem{Constantin_2011}
\leavevmode\vrule height 2pt depth -1.6pt width 23pt, {\em Periodic traveling
  gravity water waves with discontinuous vorticity}, Archive for Rational
  Mechanics and Analysis, 202 (2011), pp.~133--175.

\bibitem{Constantin2016}
{\sc A.~Constantin, W.~Strauss, and E.~V{\u{a}}rv{\u{a}}ruc{\u{a}}}, {\em
  Global bifurcation of steady gravity water waves with critical layers}, Acta
  Mathematica, 217 (2016), pp.~195--262.

\bibitem{DubreilJacotin34}
{\sc M.~L. Dubreil-Jacotin}, {\em Sur la d{\'e}termination rigoureuse des ondes
  permanentes p{\'e}riodiques d'ampleur finite.}, J. Math.\ Pures Appl.\/, 13
  (1934), pp.~217--291.

\bibitem{Cokelet1977a}
{\sc C.~E.D.}, {\em Steep gravity waves in water of arbitrary uniform depth},
  Philosophical Transactions of the Royal Society of London. Series A,
  Mathematical and Physical Sciences, 286 (1977), pp.~183--230.

\bibitem{Fraenkel2006}
{\sc L.~E. Fraenkel}, {\em A constructive existence proof for the extreme
  stokes wave}, Archive for Rational Mechanics and Analysis, 183 (2006),
  pp.~187--214.

\bibitem{GilbargTrudinger01}
{\sc D.~Gilbarg and N.~S. Trudinger}, {\em Elliptic partial differential
  equations of second order}, Classics in Mathematics, Springer-Verlag, Berlin,
  2001.
\newblock Reprint of the 1998 edition.

\bibitem{Groves2008}
{\sc M.~Groves and E.~Wahl{\'{e}}n}, {\em Small-amplitude stokes and solitary
  gravity water waves with an arbitrary distribution of vorticity}, Physica D:
  Nonlinear Phenomena, 237 (2008), pp.~1530--1538.

\bibitem{Hur07}
{\sc V.~M. Hur}, {\em Symmetry of steady periodic water waves with vorticity},
  Philos. Trans. R. Soc. Lond. Ser. A, 365 (2007), pp.~2203--2214.

\bibitem{Keady1975}
{\sc G.~Keady and J.~Norbury}, {\em Water waves and conjugate streams}, Journal
  of Fluid Mechanics, 70 (1975), pp.~663--671.

\bibitem{KeadyNorbury78}
{\sc G.~Keady and J.~Norbury}, {\em On the existence theory for irrotational
  water waves}, Math. Proc. Cambridge Philos. Soc., 83 (1978), pp.~137--157.

\bibitem{Kozlov2007}
{\sc V.~Kozlov and N.~Kuznetsov}, {\em Bounds for arbitrary steady gravity
  waves on water of finite depth}, Journal of Mathematical Fluid Mechanics, 11
  (2007), pp.~325--347.

\bibitem{Kozlov2009}
\leavevmode\vrule height 2pt depth -1.6pt width 23pt, {\em Fundamental bounds
  for steady water waves}, Mathematische Annalen, 345 (2009), pp.~643--655.

\bibitem{Kozlov2012}
\leavevmode\vrule height 2pt depth -1.6pt width 23pt, {\em Bounds for steady
  water waves with vorticity}, Journal of Differential Equations, 252 (2012),
  pp.~663--691.

\bibitem{Kozlov2014}
\leavevmode\vrule height 2pt depth -1.6pt width 23pt, {\em Dispersion equation
  for water waves with vorticity and stokes waves on flows with
  counter-currents}, Archive for Rational Mechanics and Analysis, 214 (2014),
  pp.~971--1018.

\bibitem{Kozlov2014a}
{\sc V.~Kozlov, N.~Kuznetsov, and E.~Lokharu}, {\em Steady water waves with
  vorticity: an analysis of the dispersion equation}, Journal of Fluid
  Mechanics, 751 (2014).

\bibitem{Kozlov2015}
\leavevmode\vrule height 2pt depth -1.6pt width 23pt, {\em On bounds and
  non-existence in the problem of steady waves with vorticity}, Journal of
  Fluid Mechanics, 765 (2015).

\bibitem{Kozlov2017a}
\leavevmode\vrule height 2pt depth -1.6pt width 23pt, {\em On the
  {B}enjamin{\textendash}{L}ighthill conjecture for water waves with
  vorticity}, Journal of Fluid Mechanics, 825 (2017), pp.~961--1001.

\bibitem{KozLokhWheeler2020}
{\sc V.~Kozlov, E.~Lokharu, and M.~H. Wheeler}, {\em Nonexistence of
  subcritical solitary waves},  (2020).

\bibitem{Krasovskii60}
{\sc J.~P. Krasovski{\u\i}}, {\em The theory of steady-state waves of large
  amplitude}, Soviet Physics Dokl., 5 (1960), pp.~62--65.

\bibitem{Lieberman1986}
{\sc G.~M. Lieberman and N.~S. Trudinger}, {\em Nonlinear oblique boundary
  value problems for nonlinear elliptic equations}, Transactions of the
  American Mathematical Society, 295 (1986), pp.~509--509.

\bibitem{Lokharu2020a}
{\sc E.~Lokharu}, {\em Improved bound in the {B}enjamin and {L}ighthill
  conjecture}, arXiv:2008.07312.

\bibitem{Lokharu2020}
\leavevmode\vrule height 2pt depth -1.6pt width 23pt, {\em Nonexistence of
  steady waves with negative vorticity}, arXiv:2005.08666.

\bibitem{Lokharu2020b}
\leavevmode\vrule height 2pt depth -1.6pt width 23pt, {\em A sharp version of
  the {B}enjamin and {L}ighthill conjecture about steady water waves with
  vorticity}, preprint,  (2020).

\bibitem{McLeod1997}
{\sc J.~B. McLeod}, {\em The {S}tokes and {K}rasovskii conjectures for the wave
  of greatest height}, Studies in Applied Mathematics, 98 (1997), pp.~311--333.

\bibitem{Plotnikov82}
{\sc P.~I. Plotnikov}, {\em Justification of the {S}tokes conjecture in the
  theory of surface waves}, Dinamika Sploshn. Sredy,  (1982), pp.~41--76.

\bibitem{Plotnikov2004}
{\sc P.~I. Plotnikov and J.~F. Toland}, {\em Convexity of stokes waves of
  extreme form}, Archive for Rational Mechanics and Analysis, 171 (2004),
  pp.~349--416.

\bibitem{starr}
{\sc V.~P. Starr}, {\em Momentum and energy integrals for gravity waves of
  finite height}, J. Mar. Res., 6 (1947), pp.~175--193.

\bibitem{Stokes80}
{\sc G.~G. Stokes}, {\em Considerations relative to the greatest height of
  oscillatory irrotational waves which can be propogated without change of
  form}, Mathematical and Physical Papers, 1 (1880), pp.~225--228.

\bibitem{StraussWheeler16}
{\sc W.~A. Strauss and M.~H. Wheeler}, {\em Bound on the slope of steady water
  waves with favorable vorticity}, Archive for Rational Mechanics and Analysis,
  222 (2016), pp.~1555--1580.

\bibitem{Toland1978a}
{\sc J.~Toland}, {\em On the existence of a wave of greatest height and
  {S}tokes's conjecture}, Proceedings of the Royal Society of London. A.
  Mathematical and Physical Sciences, 363 (1978), pp.~469--485.

\bibitem{Varholm2019}
{\sc K.~Varholm}, {\em Global bifurcation of waves with multiple critical
  layers}.

\bibitem{Varvaruca09}
{\sc E.~Varvaruca}, {\em On the existence of extreme waves and the {S}tokes
  conjecture with vorticity}, J. Differential Equations, 246 (2009),
  pp.~4043--4076.

\bibitem{Varvaruca2011}
{\sc E.~Varvaruca and G.~S. Weiss}, {\em A geometric approach to generalized
  stokes conjectures}, Acta Mathematica, 206 (2011), pp.~363--403.

\bibitem{Varvaruca2012}
\leavevmode\vrule height 2pt depth -1.6pt width 23pt, {\em The stokes
  conjecture for waves with vorticity}, Annales de l{\textquotesingle}Institut
  Henri Poincare (C) Non Linear Analysis, 29 (2012), pp.~861--885.

\bibitem{Wheeler15b}
{\sc M.~H. Wheeler}, {\em The {F}roude number for solitary water waves with
  vorticity}, J. Fluid Mech., 768 (2015), pp.~91--112.

\bibitem{Zufiria1987}
{\sc J.~A. Zufiria}, {\em Weakly nonlinear non-symmetric gravity waves on water
  of finite depth}, Journal of Fluid Mechanics, 180 (1987), p.~371.

\end{thebibliography}
\end{document}